\pdfoutput=1
\documentclass[11pt,a4paper,%
               abstract=on,%
               bibliography=totoc]{scrartcl}

\usepackage[utf8]{inputenc}
\usepackage[T1]{fontenc}
\usepackage[USenglish]{babel}

\usepackage{amsmath}
\usepackage{amssymb}
\usepackage{amsthm}
\usepackage{mathtools}
\usepackage[mathscr]{eucal}

\usepackage{tikz}
\usepackage{tikzpagenodes}
\usetikzlibrary{arrows.meta,babel,calc,cd,fit,quotes}
\usepackage{xcolor}
\usepackage{graphicx}

\usepackage{calc}
\usepackage{ifmtarg}

\usepackage[autostyle]{csquotes}
\usepackage[shortlabels]{enumitem}
\usepackage[lowtilde]{url}

\usepackage[backend=biber%
            , style=alphabetic%
            , giveninits=true%
            , useprefix=true%
            , maxbibnames=10%
            , maxalphanames=4%
            , urldate=long%
            , url=true%
            , isbn=false%
            , doi=false%
            , backref=true%
            , backrefstyle=none%
            , autopunct=false%
           ]{biblatex}

\usepackage[colorlinks=true,allcolors=colorForLinks,% default color for links
            citecolor=.,% use current color; we set this below for the whole [...] construct
            urlcolor=colorForUrls,%
            pdfborder={0 0 0},pdfborderstyle=,% prevent link borders
            linktoc=all,% text and page number links in the table of contents
            bookmarks,bookmarksnumbered=false,bookmarksdepth=3,% use pdf bookmarks
            bookmarksopen=true,bookmarksopenlevel=2,% instruct pdf viewers to expand one level
            pdftitle={Abstract Excision and l1-Homology},%
            pdfauthor={Johannes Witzig}]{hyperref}
\usepackage{cleveref}

\usepackage{helpers} % local helpers.sty, provides \newlet, \defmathbbsymbol(sub)s

\colorlet{colorForCitations}{black!60}
\definecolor{colorForLinks}{HTML}{00005A}
\definecolor{colorForUrls}{HTML}{474770}
\colorlet{colorForBackrefLinks}{colorForLinks!75}

\AtEveryCite{\color{colorForCitations}}

\addtokomafont{sectioning}{\boldmath}

\DeclareSourcemap{
  \maps[datatype=bibtex, overwrite]{
    \map{
      \pernottype{article}
      \step[fieldset=series, null]
      \step[fieldset=volume, null]
    }
    \map{
      \pertype{book}
      \step[fieldset=pages, null]
    }
  }
}

\DeclareFieldFormat*{title}{\mkbibitalic{#1}}
\DeclareFieldFormat*{journaltitle}{#1\isdot}

\renewbibmacro*{pageref}{% adapted from biblatex.def
    \iflistundef{pageref}
    {}
    {\setunit{\adddot\addspace}\printtext{%
        {\small\color{black!75}\hypersetup{linkcolor=colorForBackrefLinks}%
        Cited on page%
        \ifnumgreater{\value{pageref}}{1}{s}{}%
        \ppspace%
        \printlist[pageref][-\value{listtotal}]{pageref}}%
        }%
    }%
}

\DeclareLabelalphaTemplate{
    \labelelement{
        \field[final]{shorthand}
        \field{label}
        \field[ifnames=1,strwidth=2,strside=left]{labelname}
        \field[varwidthlist]{labelname}
    }
}
\DeclareFieldFormat{extraalpha}{#1}

\setlist[itemize,1]{label=--}
\setlist[itemize,2]{label=\textasteriskcentered}
\setlist[enumerate,1]{ % make enumerated lists use (i), (ii), ... by default
                       label=(\roman*)
                     , itemsep={2.5pt plus 1.0pt minus 1.0pt}
                     }

\tikzcdset{
    , matrix/.code={\tikzcdset{every matrix/.append style={#1}}}
    , inclusion/.forward to=/tikz/commutative diagrams/hook
    , pmatrix/.style={
        every matrix/.append style={
            left delimiter=(, right delimiter=),
            inner xsep=-1ex,    % see default inner sep values
            inner ysep=-0.85ex  %  in tikzlibrarycd.code.tex
        }
    }
}

\AtBeginDocument{}

\newlet\softhyph=\-
\renewcommand{\-}{\babelhyphen{hard}}

\let\colon=\egregcolon

\newcommand{\qedplacement}[2]{#1}
\newcommand{\qedNOThere}{\\*}
\newcommand{\qedafterdisplay}{\setlength{\belowdisplayskip}{2pt}}

\newcommand{\citesep}{\allowbreak\,}
\newcommand{\icite}[2][]{\citesep\cite[#1]{#2}}     % for /i/nitial citation
\newcommand{\mcite}[2][]{\allowbreak\cite[#1]{#2}}  % for /m/ultiple follow-ups

\newcommand{\titlecitesep}{\texorpdfstring{\allowbreak\,}{}}
\newcommand{\titlecite}[2]{\allowbreak\texorpdfstring{\cite[#2]{#1}}{}}
\newcommand{\citekerodon}[1]{\cite[#1]{web:kerodon}}
\newcommand{\icitekerodonX}[1]{\citesep\citekerodon{#1}}
\newcommand{\mcitekerodonX}[1]{\allowbreak\citekerodon{#1}}

\newcommand{\kerodontag}[1]{\href{https://kerodon.net/tag/#1}{#1}}
\newcommand{\citekerodontag}[1]{\citekerodon{\kerodontag{#1}}}
\newcommand{\icitekerodon}[1]{\citesep\citekerodontag{#1}}

\newcommand{\titlecitekerodon}[1]{\allowbreak\texorpdfstring{\citekerodontag{#1}}{}}

\newtheoremstyle{mythms}
 {15pt}% space above
 {12pt}% space below
 {}% body font
 {}% indent amount
 {\bfseries}% theorem head font
 {.}% punctuation after theorem head
 {0.6cm plus 0.25cm minus 0.1cm}% space after theorem head (\newline possible)
 {}% theorem head spec
 
\theoremstyle{mythms}
\newtheorem{globalnum}{DUMMY DUMMY DUMMY}[section]

\newtheorem{thDef}[globalnum]{Definition}
\newtheorem{thTheorem}[globalnum]{Theorem}
\newtheorem{thProposition}[globalnum]{Proposition}
\newtheorem{thLemma}[globalnum]{Lemma}

\newtheorem{thConvention}[globalnum]{Convention}
\newtheorem{thRemark}[globalnum]{Remark}
\newtheorem{thOutlook}[globalnum]{Outlook}
\newtheorem{thExample}[globalnum]{Example}

\newtheorem{thExamples}[globalnum]{Examples}

\newtheorem*{introMainTheorem}{Main Theorem}

\makeatletter
\newlet\origthmhead=\thmhead
\renewcommand{\thmhead}[3]{%
\origthmhead{#1}{#2}{#3}%
\belowpdfbookmark{#1\@ifnotempty{#1}{ }#2\thmnote{ (#3)}}{#1#2}%
}
\makeatother

\renewcommand{\qedsymbol}{$\blacksquare$}

\let\cref=\Cref

\newcommand*{\pcref}[1]{(\cref{#1})}

\newcommand{\crefitem}[2]{\cref{#1}\texorpdfstring{\,}{}\ref{#1#2}}

\newcommand{\envEndsSymbol}{\textcolor{gray}{$\lrcorner$}}
\newcommand{\envEnds}{%
    \begingroup%
        \renewcommand{\qedsymbol}{\envEndsSymbol}%
        \qed%
    \endgroup%
    \ignorespaces%
}
\newcommand{\envEndsHere}{%
    \tikz[remember picture,overlay]\path
        ({current page text area.east} |- {0,0})
        node [anchor=south,inner sep=0pt] {\llap{\envEndsSymbol}}
        ;
}

\newcommand*{\introParagraph}[1]{%
    \phantomsection%
    \belowpdfbookmark{#1}{#1}%
    \paragraph[beforeskip=2cm]{#1.}%
}

\RedeclareSectionCommand[beforeskip=3.25ex plus 1ex minus .5ex]{paragraph}

\newcommand{\sDMO}[1]{\expandafter\DeclareMathOperator\csname#1\endcsname{#1}}

\sDMO{colim}
\sDMO{Cone}
\sDMO{Exc}
\sDMO{Fun}
\sDMO{Id}
\sDMO{id}
\sDMO{im}
\sDMO{ker}
\sDMO{Map}
\sDMO{Mor}
\sDMO{Ob}
\sDMO{Pre}

\DeclareMathOperator{\nerve}{N}
\newcommand{\nervehc}{\nerve^{\hc}}
\newcommand{\nervedg}{\nerve^{\dg}}

\newcommand{\MakeCategoryName}[1]{%
    \expandafter\DeclareMathOperator\csname#1\endcsname{\mathsf{#1}}
}

\MakeCategoryName{Cat}
\MakeCategoryName{Ch}
\MakeCategoryName{Set}
\MakeCategoryName{sSet}
\MakeCategoryName{Top}

\newcommand{\Chk}{\Ch_k}
\newcommand{\ChR}{\Ch_{\R}}
\newcommand{\dgCh}{\Ch^{\dg}}

\let\subset=\subseteq

\DeclareMathSymbol{\cross}{\mathbin}{symbols}{"02}% same as \times in fontmath.ltx
\newcommand{\bd}{\mathsf{b}}
\newcommand{\Cl}{C^{\lone}}
\newcommand{\Cb}{C_{\bd}}
\newcommand{\dg}{\mathsf{dg}}
\newcommand{\defeq}{\coloneqq}
\newcommand{\eqdef}{\eqqcolon}
\newcommand{\Hb}{H_{\bd}}
\newcommand{\hc}{\mathsf{hc}}
\newcommand{\Hl}{H^{\lone}}
\newcommand{\incl}{\hookrightarrow}
\newcommand{\inv}{^{-1}}
\newcommand{\lone}{\ell^1}
\newcommand{\lonenorm}[2][]{\norm[#1]{#2}_1}
\newcommand{\op}{^{\mathsf{op}}}

\newcommand{\setZeroto}[1]{\{0,\ldots,#1\}}

\newcommand{\oo}{$\infty$} % for use outside of math mode

\newcommand{\after}{\surround{\mskip4mu plus 2mu minus 1mu}{\mathord{\circ}}}
\newcommand{\blank}{\surround{\mskip1.5mu plus 0mu minus 0.5mu}{\blanksymbol}}
\newcommand{\blanksymbol}{\mathord{\color{black!20}\bullet}}
\renewcommand{\equiv}{\mathrel{\simeq}}
\newcommand{\longto}{\longrightarrow}
\newcommand{\Mid}[1][\;]{#1\vert#1}
\newcommand{\pdfInfty}{\texorpdfstring{$\infty$}{infinity}}
\newcommand{\pdfLOne}{\texorpdfstring{$\lone$}{l1}}
\newcommand{\qqtextqq}[1]{\qquad\text{#1}\qquad}
\newcommand{\surround}[2]{#1#2#1}
\newcommand{\xunderrightarrow}[1]{\xrightarrow[{#1}]{}}

\newcommand{\card}[1]{\# #1} % cardinality of a set
\newcommand{\coface}[2]{d^{#1}_{#2}} % coface map in \simplexcat
\newcommand{\djunion}{\sqcup} % disjoint union
\newcommand{\dsum}{\oplus} % direct sum
\newcommand{\horn}[2]{\Lambda^{#1}_{#2}}
\newcommand{\intersect}{\cap}
\newcommand{\isom}{\cong}
\newcommand{\ordcat}[1]{[#1]} % (pre)order category of natural number
\newcommand{\point}{\ast} % one-point space
\newcommand{\restrict}[1]{|_{#1}} % function restriction
\newcommand{\simplexcat}{\mathord{\triangle}}
\newcommand{\stdsimplex}[1]{\Delta^{#1}}
\newcommand{\topstdsimplex}[1]{\Delta_{\mathsf{top}}^{#1}}
\newcommand{\union}{\cup}

\defmathbbsymbols{Z Q C}
\defmathbbsymbolsubs{N R}

\DeclareMathSymbol{\wedge}{\mathbin}{symbols}{"5F}
\newlet\boxsmash=\smash

\DeclarePairedDelimiter{\abs}{\lvert}{\rvert}
\DeclarePairedDelimiter{\norm}{\lVert}{\rVert}

\newcommand{\pushoutobj}[1]{\pushoutobjPRIV #1}
\newcommand{\pushoutobjPRIV}[5]{#1 \mathbin{{\sqcup}_{#2,#4}} #5}
\newcommand{\pushoutmor}[1]{\pushoutmorPRIV #1}
\newcommand{\pushoutmorPRIV}[4]{#1 \mathbin{{\sqcup}_{#2,#3}} #4}

\newcommand{\institute}{%
    Fakultät für Mathematik, %
    Universität Regensburg, %
    93040 Regensburg, Germany%
}

\newcommand{\email}{%
    \href{mailto:Johannes.Witzig@mathematik.uni-regensburg.de}%
         {Johannes.Witzig@mathematik.uni-regensburg.de}%
}

\title{Abstract Excision and $\lone$-Homology}
\author{Johannes Witzig\thanks{\institute\newline\email}}
\date{Version of \today}

\begin{document}
\maketitle
\thispagestyle{empty}

\begin{abstract}
    We use the abstract setting of excisive functors in the language of
    \oo-categories to show that the best approximation to the $\lone$-homology
    functor by an excisive functor is trivial.

    Then we make an effort to explain the used language on a conceptual level
    for those who do not feel at home with \oo-categories, prove that the
    singular chain complex functor is indeed excisive in the abstract sense,
    and show how the latter leads to classical excision statements in the form
    of Mayer-Vietoris sequences.
\end{abstract}

\begingroup\small
\tableofcontents
\endgroup

\enlargethispage{3cm}
\medskip\noindent%
{\usekomafont{sectioning}Acknowledgments.}\quad
The author was supported by the CRC~1085 \emph{Higher Invariants}
(Universität Regensburg, funded by the DFG).
He is grateful to
Clara Löh,
Hoang Kim Nguyen,
George Raptis and
Georg Biedermann
for helpful conversations and their comments on earlier versions of
this article. He further appreciates the detailed feedback by the
anonymous referee.

% start at section number 0 for the introduction
\setcounter{section}{-1}

\section{Introduction}
In order to efficiently compute the value of a (generalized) homology theory
on a topological space, a divide and conquer approach is usually most
convenient: Given an open cover $X = U \cup V$ of a space~$X$, we want to
infer the homology of the latter from the homologies of $U$, $V$ and~$U\cap V$.
The existence of a link between these objects is guaranteed by the
\emph{Mayer-Vietoris sequence}, which in turn is a formal consequence
of the Eilenberg-Steenrod axioms\icite[Sec.~10.7]{b:tomdieck08},
in particular the \emph{Excision Axiom}.
For example, singular homology fulfills the latter, which is the content
of the classical \emph{Excision Theorem}%
\icite[Sec.~9.4]{b:tomdieck08}\mcite[Thm.~2.20]{b:hatcher02}.

It is known that \emph{bounded cohomology}~$\Hb^*$ (of spaces) and its
sibling \emph{$\lone$\-homology}~$\Hl_*$ do \emph{not} satisfy excision. For example,
$\Hl_k(S^1) \isom 0$ for $k\in\N[\geq1]$, but $\Hl_2(S^1 \wedge S^1)
\not\isom 0$\icite[Ex.~(2.8)]{p:loeh08}. One can, however, ask the question
whether we can somehow enforce those functors to be excisive. In other words,
can we find a best approximation to those functors (in some precise formal
sense) that \emph{is} excisive?

To answer this question in at least one possible way, we use the abstract and
broadly applicable notion of excisive functors in the sense of Goodwillie
calculus. The natural framework for this abstract setting is the language of
\oo\-categories, which we will also use where appropriate. We show:
\begin{introMainTheorem}
    The best excisive approximation to the $\lone$\-chain complex functor
    $\Cl\colon\Top_*\to\ChR$ from the right is trivial (i.e. mapping everything
    to~$0$).
\end{introMainTheorem}
The precise statement and proof will be given below as \cref{thm:l1approx}. In order
to import the 1-categorical $\lone$\-chain complex functor to the world of
\oo\-categories, we use the localization machinery from Cisinski's
book\icite{b:cisinski19}, see \cref{rem:localization}.

\introParagraph{Context: \pdfLOne-homology and bounded cohomology}
A basic object in algebraic topology is the \emph{singular chain complex~$C(X,A)$ with real
coefficients} of a pair of topological spaces~$(X,A)$. Whereas, classically, we
take its homology, or we dualize and take cohomology (yielding \emph{singular
(co)homology}), we can also obtain a more functional analytic theory: First, we
equip $C_n(X)$ with the \emph{$\lone$\-norm} associated to the canonical basis,
i.e., we set
\[ \lonenorm[\Big]{ \sum\nolimits_{j=0}^N a_j\,\sigma_j } \defeq
    \sum\nolimits_{j=0}^N \mkern1.5mu\abs{a_j} \quad\in\R[\geq0]
\]
for reduced chains in~$C_n(X)$, and $C_n(X,A)$ with the quotient norm. The
boundary operator of~$C(X,A)$ is bounded with respect to the induced operator
norms, hence we can take the norm-completion of the whole complex and obtain
the \emph{$\lone$\-chain complex $\Cl(X,A)$}.  Its homology is the
\emph{$\lone$\-homology of~$(X,A)$}, the topological dual of (either $C(X,A)$
or) $\Cl(X,A)$ is the \emph{bounded cochain complex~$\Cb(X,A)$ of~$(X,A)$}, and
cohomology of the latter is the \emph{bounded cohomology of~$(X,A)$}.
For details and (basic) properties, we refer the reader to the literature%
\icite[Sec.~3.1]{p:loeh08}%
\mcite[Introduction and Ch.~2]{pp:loeh07}.

The prime example of an invariant defined in terms of~$\lonenorm{\blank}$ is the
\emph{simplicial volume}\icite{web:map:simplicialvolume}, which was introduced
by Gromov\icite{p:gromov82} and measures how difficult it is to build the
fundamental class of a manifold out of singular simplices with real
coefficients.
It is linked to Riemannian geometry\icite{p:gromov82} and -- except for a few
cases -- very hard to compute. To make structural statements about simplicial
volume, the theory of bounded cohomology and $\lone$\-homology has been
indispensable so far.

Bounded cohomology also has various connections to other mathematical areas
(some of which are summarized in Monod's nice \enquote{invitation to bounded
cohomology}\icite{p:monod06}), so in view of the (albeit imperfect)
dualism between $\lone$\-homology and bounded cohomology, it makes sense to
obtain a more throughout understanding of the former as well.
One should also point out, that some computations in bounded cohomology
\emph{cannot} be dualized to $\lone$-homology. For example, certain
transfer maps can only be constructed in the cohomological
setting\icite[Caveats 5.8 and~5.6]{p:loeh_Isosl1}.

To get the appropriate $\lone$\-chain complex functor between \oo\-categories via
the universal property of the localization \pcref{rem:localization},
it is important to recognize $\Cl$ as a \emph{relative} functor
\pcref{def:relcat}, i.e., mapping weak (homotopy) equivalences of topological
spaces to quasi-isomorphisms of chain complexes. This follows from the
fact that $\Cb$ is a relative functor\icite[6.4~Cor.]{pp:ivanov17} and the
translation principle by Löh\icite[Cor.~5.1]{p:loeh08}. (Note, that Ivanov's
result is (maybe for convenience) only stated for maps between
\emph{path-connected} spaces, but his proof of the theorem preceding the
corollary can be taken over verbatim if one drops the path-connectedness
assumptions and replaces the \enquote{i.e.}-part by the more general definition
of \enquote{$k$\-equivalence} involving all base points and~$\pi_0$; see the
literature\icite[Sec.~6.7]{b:tomdieck08}\mcite[Ch.~4]{b:hatcher02}.)

\introParagraph{Variations of the main theorem}
By taking opposite \oo\-categories appropriately, basically the same proof
as for \cref{thm:l1approx} shows that the best excisive approximation to
the bounded cochain complex functor $\Top_*\op \to \ChR$ \emph{from the left}
is also trivial. (One must be careful with the terminology, though: in the
context of contravariant functors, one has to read \cref{def:excisive} as
\enquote{pushout squares \emph{in the original category}}, i.e. before
applying~${}\op$.)

Dually to our investigation, one might also ask for an excisive
approximation of the $\lone$\-chain complex functor \emph{from the left}
and of the bounded cochain complex functor \emph{from the right}.
It is shown by Raptis\icite{pp:raptis21}, that the latter is exactly the
comparison map from the bounded to the singular cochain
complex\icite[Sec.~2.5]{pp:raptis21}, and that the arguments also
dualize to the comparison map in the case of
$\lone$\-homology\icite[Sec.~3.3]{pp:raptis21}.

As for approximations to bounded cohomology and $\lone$-homology
\emph{of groups}, the setup of this article does not seem to be
very useful: One fundamental problem is, that pushouts
are well suited for decompositions of topological spaces,
but not in the case of groups, where they only correspond to
amalgamated free products.

Looking deeper into Goodwillie calculus reveals, that a functor being
excisive is only the first of a whole sequence of (increasingly weaker)
conditions on a functor. In generalization of the main theorem of this
article, the author showed, that the $n$\-excisive
approximation to the $\lone$\-chain complex functor is trivial for
\emph{all} $n\in\N$. The details, which become a lot more technical,
can be found in the author's PhD~thesis\icite[Sec.~6.3]{pp:witzig_phd}.

\introParagraph{Notations and conventions} \label{sec:conventions}%
As we want to work with functor categories, it is necessary to avoid set
theoretic issues. 
The reader, who does usually not worry about size issues, may find
an explanation of the problem and strategies for its solution in
\cref{rem:sizeissues}.

If one intends to do general treatment of \oo\-category theory itself,
the best way to solve this, is to parametrize each and every statement
over the meaning of \enquote{small} and \enquote{large}. This is, what
happens in (the second half of) Cisinski's book\icite{b:catLR}, and it
has the benefit of greatest flexibility.
However, as we mainly want to use \oo\-categories to treat explicit
examples and to survey some aspects of the theory, this approach would add
too much complexity and most certainly confusion among novice readers.
For these reasons, we will choose a size hierarchy now, into which all
objects of this article fit.

We fix an increasing chain of four Grothendieck universes,
whose elements we refer to as \emph{small}, \emph{large}, \emph{very large}
and \emph{huge}. We occasionally use the term \emph{class} for a subset of
the small universe; hence, each class is a large set.
We then denote the category of
\begin{itemize}[nosep]
\item  small topological spaces by~$\Top$,
\item  large sets by~$\Set$, and
\item  large, but locally small 1-categories by~$\Cat$.
\end{itemize}
All the usual categories of algebraic theories (rings, modules,~\dots) live
at the same stage as~$\Top$, and are objects of~$\Cat$.
Rings are always unital. For a ring~$k$, we denote the
category of chain complexes of left $k$-modules by~$\Chk$.

For a category~$A$ and objects $x,y$ of~$A$, we let $\Mor_A(x,y)$ denote
the set of morphisms $x\to y$; more generally (and if $A$~is locally large),
$\Mor_A$ denotes the corresponding functor $A\op\cross A\to\Set$, where
$A\op$ denotes the opposite category of~$A$.

When used as an \oo\-category,\label{par:localizationConvention} the symbols $\Top$ and~$\Chk$ need to
be re-interpreted as the localization \pcref{rem:localization} of their
1-categorical counterpart with respect to weak homotopy equivalences and
quasi-isomorphisms, respectively. We do \emph{not} employ the common
practice of using the nerve functor \pcref{def:1nerve} implicitly.

\introParagraph{Organization of this article}
In \cref{sec:approximation} we lay out the results that we are able to obtain
via the Goodwillie calculus framework. \cref{sec:glimpse} contains a very
brief introduction to \oo\-categories in the sense of Lurie\icite{b:HTT} and
tries to provide some insight on a conceptual level to a few aspects of the
theory; we would like to provide as much \enquote{glue} as needed for the
\enquote{1-categorical reader} who is usually not working with \oo\-categories.
In \cref{sec:singular} we provide a full proof of the fact that the
singular chain complex functor (on the level of \oo\-categories) is excisive in
our abstract sense \pcref{def:excisive}. Finally, in \cref{sec:classical}, we
show how to obtain Mayer-Vietoris sequences from \oo\-categorical pullback
squares, and in particular how to obtain the classical Mayer-Vietoris sequence
for singular homology in this way.

\introParagraph{How to read this article}
Readers with at least some prior knowledge of \oo\-category theory can start at
\cref{sec:approximation}, might want to skip most of \cref{sec:glimpse}, and
then (optionally) continue with \cref{sec:singular} and \cref{sec:classical}.
The reader, who is proficient in model category theory, can skip
\cref{sec:modelcatbits}.

For the reader without \oo\-categorical background, we recommend to begin with
\cref{sec:glimpse}, ignoring \cref{sec:othernames} on the first
visit. Depending on the familiarity with modern homotopy theory, one might
also want to read \cref{sec:modelcatbits} to get a feeling for the nature
of homotopical situations. Afterwards, the remaining parts can be consumed
as desired, e.g. linearly from \cref{sec:approximation}.

\section{Excisive approximation} \label{sec:approximation}%
In this section, we investigate the excisive approximation of
$\lone$\-homology -- or more precisely, of the $\lone$\-chain complex functor.
First, we introduce a concept of \emph{excision} whose formulation applies to
\emph{any} functor between \oo\-categories \pcref{def:excisive},
we quote a theorem by Lurie ensuring the existence of such an approximation
\pcref{thm:P1}, and we explain in what sense such an approximation is universal
\pcref{rem:bestapprox}.
We then use these notions to show that the excisive approximation
of~$\Cl\colon\Top_*\to\ChR$, i.e. the $\lone$\-chain complex functor on pointed
spaces, is trivial \pcref{thm:l1approx}.
An explicit formula by Lurie makes the argument short and accessible.

As this section's setup is entirely \oo\-categorical, the word
\enquote{category} will always mean \oo\-category in the sense of
Lurie\icite{b:HTT} \pcref{def:infcat}
and we will use $1$\-category to refer to categories
in the sense of traditional category theory.
Likewise, we will deal with other concepts such as functors, limits, etc.,
i.e. \enquote{functor} will mean a functor of \oo\-categories,
\enquote{(co)limit}/\enquote{pushout}/\enquote{pullback} will mean the
\oo\-categorical notion, etc.
In particular, we remind the reader of our convention that $\Top$ and~$\ChR$
become \oo\-categories via localization (page~\pageref{par:localizationConvention}).
We highlight these choices by the following:
\begin{thConvention}
    In this section, and this section only,
    the term \emph{category} refers to \emph{\oo-category}.
\end{thConvention}

\subsection{Abstract excision}
First of all, we introduce the concept of an excisive functor of \oo\-categories:
\begin{thDef}[excisive functor] \label{def:excisive}%
    Let $C$ be a category with pushouts and let $F\colon C\to D$ be a functor.
    Then $F$ is \emph{excisive} if it sends pushout squares to pullback squares,
    i.e., if the following holds: Given a pushout square
    \[ \begin{tikzcd}
        W \ar[r, "f" midway] \ar[d, "g"'] & U \ar[d, "g'"]
        \\
        V \ar[r, "f'" midway] & X
        \end{tikzcd}
    \]
    in~$C$, the diagram
    \[ \begin{tikzcd}
        F(W) \ar[r, "F(f)"] \ar[d, "F(g)"'] & F(U) \ar[d, "F(g')"] \\
        F(V) \ar[r, "F(f')"] & F(X)
        \end{tikzcd}
    \]
    is a pullback square in~$D$.%
    \envEnds
\end{thDef}
Note again, that this must be read in the \oo\-categorical setting. In
particular, pushout/\hspace{0pt}pullback squares do \emph{not} have to commute \enquote{on
the nose} but only \enquote{up to homotopy} and the latter is part of the
data(!), see also \cref{ex:commutativesquares} and \cref{rem:sloppynotation}.

\begin{thExample} \label{ex:nonexcisivefunctor}%
    The constant functor $\Top\to\Top$ that sends everything to a point is
    excisive. On the other hand, the identity functor on~$\Top$ is
    \emph{not} excisive: for example, there is a pushout square
    \[ \begin{tikzcd}[column sep=1.4cm, row sep=0.9cm]
            S^0 \ar[r, "\text{inclusion}"]\ar[d]
            & D^1\ar[d, "\text{\parbox{1.4cm}{\raggedleft quotient\\[-0.5\baselineskip]map}}"']
            \\
            \point\ar[r]
            & D^1/S^0
        \end{tikzcd}
    \]
    that is \emph{not} a pullback square. We will get back to this in
    \cref{ex:nonexcisiveproof}.
\end{thExample}

As it turns out, given a functor between categories with enough limits,
it always has a \enquote{best approximation} by an excisive functor in a precise
sense. The following definition captures the niceness assumptions that we have
to impose on the codomain category:
\begin{thDef} \label{def:differentiable}%
    Let $C$ be a category. Then $C$ is \emph{differentiable} if it admits finite
    limits and sequential colimits and if the formation of the latter commutes
    with the former.
    
    More explicitly, this means:
    every finite\icitekerodon{0130} diagram $K\to C$ admits a limit,
    every diagram $\nerve(\N)\to C$ admits a colimit in~$C$,
    and the functor $\colim\colon\Fun(\nerve(\N),C)\to C$ commutes with finite limits.
    Here, $\N$~is viewed as the preorder category
    (\crefitem{def:simplicialset}{:preordercat}) of the usual ordering on the
    natural numbers and $\nerve$~is the nerve functor \pcref{def:1nerve} from
    $1$\-categories to \oo\-categories.
\end{thDef}

\begin{thTheorem}[\titlecite{b:HA}{Thm.~6.1.1.10 and Ex.~6.1.1.28}] \label{thm:P1}%
    Let $C$ be a category with a terminal object and finite colimits and let
    $D$ be a differentiable category. Let $\Exc(C,D)$ be the full subcategory of
    the functor category $\Fun(C,D)$, spanned by the excisive functors. Then the
    inclusion functor $\Exc(C,D)\incl\Fun(C,D)$ has a left adjoint
    $P_1\colon\Fun(C,D)\to\Exc(C,D)$.
    
    Furthermore, if $F\colon C\to D$ is \emph{reduced}, which means that $F$
    maps every terminal object of~$C$ to a terminal object of~$D$,
    and if $D$~has a zero object, the following holds:
    \[ P_1F \equiv \colim_{n\in\N} \Omega^n\after F\after\Sigma^n \]
    where $\Omega$ and $\Sigma$ are the loop and suspension functor on $D$
    and~$C$, respectively.%
    \envEnds
\end{thTheorem}
It should be noted, that the functor categories that appear in the theorem are
usually not locally small anymore. Though the existence of an adjoint as well as
the statement that two given functors are adjoint to each other do \emph{not}
depend on the ambient universe\icite[Rem.~6.1.12 and Thm.~6.1.23\,(v)]{b:catLR},
the use of mapping space functors (\enquote{Hom space} in Cisinski's
book\icite[Sec.~5.8]{b:catLR}) will be convenient in the following remark.
In order to meet the smallness assumptions for this, it is necessary to view
the functor categories from our very large universe. For more discussion of mapping
spaces, see \cref{sec:Map} (with the small caveat, that our current shift of
universes implies that the targets of mapping space functors in
\cref{rem:bestapprox} are not~$\Top$, but the category of \emph{large}
topological spaces).

For any functor $F\colon C\to D$, the functor $P_1F\colon C\to D$ can be seen
as a best approximation of~$F$ by an excisive functor from the right, as we
shall now explain:
\begin{thRemark}[best approximation through adjunction] \label{rem:bestapprox}%
    One way of making the term \emph{best approximation} precise in a categorical
    setting is by so-called \emph{universal arrows}\icite[Sec.~III.1]{b:maclane98},
    i.e. by a certain mapping property:

    Let $C'$ be a full subcategory of some category~$C$ and let $X$ be an object
    of~$C$ that we want to approximate by an object of~$C'$. For example,
    $C$~could be a functor category and $C'$~could be its full subcategory of
    excisive functors. Then a \emph{best approximation of~$X$ by an object
    of~$C'$ from the right} consists of an object $Y$ of~$C'$ and a morphism
    $f\colon X\to Y$ in~$C$ with the following property:
    for every other morphism $f'\colon X\to Y'$ to an object of~$C'$, there
    exists an (up to homotopy) unique morphism $g\colon Y\to Y'$ such that
    \[ \begin{tikzcd}
        X \ar[r, "f"] \ar[rd, "f'"'] & Y \ar[d, "g"] \\ & Y'
        \end{tikzcd}
    \]
    commutes. Intuitively, we search for the \enquote{biggest quotient} of~$X$
    that has the desired properties.
    Of course, the whole paragraph could be dualized, by which we would obtain
    approximations \emph{from the left}, i.e. \enquote{biggest subobjects}.

    It is a classical 1-categorical fact\icite[Thm.~IV.1.1]{b:maclane98}
    that an adjunction gives rise to universal arrows. We will sketch the
    argument for our situation:
    \newcommand{\FUN}{\Fun}\newcommand{\EXC}{\Exc}%
    \newcommand{\MapF}{\Map_{\FUN}}\newcommand{\MapE}{\Map_{\Exc}}%
    \newcommand{\PF}{P\!F}%
    Assume that we have the inclusion functor $i\colon\EXC\incl\FUN$ of a full
    subcategory and that it has a left adjoint~$P\colon\FUN\to\EXC$.
    Using mapping space functors, one can define adjunctions in \oo\-categories
    analogously to the 1-categorical case\icite[Def.~6.1.3]{b:cisinski19}, i.e.
    we have an invertible natural transformation
    \[ \MapF(\blank, i(\blank)) \overset{c}\longto \MapE(P(\blank), \blank)
    . \]
    Here, invertibility means, that each component of~$c$ is an invertible
    morphism\icite[Def. 1.6.10 and~1.5.1]{b:catLR}, or equivalently, that~$c$
    is an invertible morphism in the corresponding functor
    category\icite[Cor.~3.5.12]{b:catLR}.

    Now let $F$ be an object of~$\FUN$, choose an inverse of~$c$ and let
    $\eta \defeq c\inv_{F,\PF}(\id_{\PF}) \in\MapF(F,i\PF)$. In order to show that
    this morphism
    \[ F \overset{\eta}\longto \PF \]
    is universal in the above sense, let $P'$ be another object of~$\EXC$, let
    $(\eta'\colon F\to P') \in\MapF(F,iP')$ be a morphism and set
    $\theta \defeq c_{F,P'}(\eta')$. Then the triangle
    \[ \begin{tikzcd}
        F \ar[r, "\eta"] \ar[rd, "\eta'"'] & \PF \ar[d, "\theta"] \\ & P'
        \end{tikzcd}
    \]
    commutes, which can be seen by chasing~$\eta$ through the following diagram:
    \[ \begin{tikzcd}[column sep=3cm, row sep=20pt]
            \MapF(F,\PF) \ar[r, "{\MapF(\id_F, \theta)}"] \ar[d, "{c_{F,\PF}}"]
            & \MapF(F,\PF) \ar[d, "{c_{F,P'}}"]
            \\ \MapE(\PF,\PF) \ar[r, "{\MapE(\id_{\PF}, \theta)}"]
            & \MapE(\PF, P') \rlap{.}
            \envEndsHere
        \end{tikzcd}
    \]
\end{thRemark}

\subsection{Excisive approximation of the \pdfLOne-chain complex functor}
\label{sec:approxchaincomplexfunctor}% TODO: remove redundant label
\label{sec:Approxl1ChainComplexFunctor}%
We can now apply \cref{thm:P1} to $\lone$\-homology, but we have to be a little
bit careful: After a bit of thought, it is quite apparent that the classical
\emph{excision axiom} for a homology theory in the sense of Eilenberg and
Steenrod is \emph{not} a property of a single functor~$H_n$ for a fixed~$n$,
but of a whole sequence of functors: after all, we know\icite[cf. Ch.~7, esp.
7.34 and 7.35]{b:switzer02}
that excision can equivalently be recast in the form of a Mayer-Vietoris
sequence running through all \enquote{dimensions} of a homology theory.
(In the usual form of the Eilenberg-Steenrod axioms the excision axiom
seems to depend only on a single functor at a time, but this is deceptive because it
uses \emph{relative} homology which is tightly linked to its absolute version in
different(!) dimensions by the long exact sequence of a pair.)

So the upshot of this discussion is that we should not try to approach each of the
functors~$\Hl_n$ individually, but instead we have to look at the corresponding
functor $\Top\to\ChR$ to chain complexes. We will see later \pcref{sec:singular}
that indeed the usual singular chain complex functor \emph{is} excisive in our
sense.
In order to apply the formula of \cref{thm:P1}, we need a reduced functor, but
$\Cl(\point)$ is \emph{not} a contractible chain complex. This is why we
consider pointed spaces and the functor $\Cl\colon\Top_*\to\ChR$.
We then obtain:

\begin{thTheorem}[excisive approximation of \pdfLOne-homology]
    \label{thm:l1approx}%
    There exists a best excisive approximation of~$\Cl\colon\Top_*\to\ChR$ from the right,
    i.e. an excisive functor $P\colon\Top_*\to\ChR$ and a natural transformation
    $\eta\colon\Cl\to P$,
    such that for every other excisive functor $P'\colon\Top_*\to\ChR$ and natural
    transformation $\eta'\colon\Cl\to P'$
    there exists a natural transformation $\theta\colon P\to P'$ (unique up to
    homotopy) that makes the triangle
    \[ \begin{tikzcd}
        \Cl \ar[r, "\eta"] \ar[rd, "\eta'"'] & P \ar[d, "\theta"] \\ & P'
        \end{tikzcd}
    \]
    commutative.
    This best approximation is trivial, i.e. $P(X) \equiv 0$ for all pointed
    spaces~$X$.
\end{thTheorem}

\begin{proof}
    The category~$\Top_*$ has all small limits and colimits. Furthermore, $\ChR$~is
    stable\icite[Prop.~1.3.5.9 (with Prop.~1.3.5.15)]{b:HA} and
    has all small colimits, so it is in particular
    differentiable\icite[Ex.~6.1.1.7]{b:HA}.
    Hence, the prerequisites of \cref{thm:P1} are fulfilled and the functors
    $P_1\colon\Fun(\Top_*,\ChR)\to\Exc(\Top_*,\ChR)$ and
    $P_1(\Cl)\colon\Top_*\to\ChR$ exist. Since $P_1$ is left adjoint to the
    inclusion functor $i\colon\Exc(\Top_*,\ChR)\hookrightarrow\Fun(\Top_*,\ChR)$, we also
    know (see \cref{rem:bestapprox}) that the unit morphism
    $F = \Id(F)\to (i\after P_1)(F)$ of the adjunction
    provides a universal arrow in the sense described in the
    claim.
    
    Furthermore, since we consider $\Cl$ relative to the base point, the
    second part of \cref{thm:P1} is applicable. For a pointed space~$X$
    this gives us
    \[ P_1(\Cl)(X) \equiv \colim_n \Omega^n\bigl( \Cl (\Sigma^n X) \bigr)  ,\]
    where $\Sigma\colon\Top_*\to\Top_*$ is the (unreduced) suspension functor and
    $\Omega\colon\ChR\to\ChR$ is the loop functor on chain complexes. But since
    $\Sigma^nX$ is simply-connected for $n\in\N[\geq2]$ and we know that $\Cl$
    vanishes on path-connected spaces with amenable (so in particular trivial) fundamental
    group\icite[8.4~Thm.]{pp:ivanov17}\mcite[Cor.~5.1]{p:loeh08},
    we get $P_1(\Cl)(X) \equiv \colim_{n\in\N[\geq2]} \Omega^n(0)$.
    Consider the fact that the loop functor~$\Omega$ on chain complexes can
    be realized by just shifting a complex down by one (i.e. $(\Omega C)_n =
    C_{n+1}$). Using this, it follows that we have
    \[ P_1(\Cl)(X) \equiv \colim_{n\in\N[\geq2]} 0 \equiv 0
    .
    \qedplacement{\qedafterdisplay}{\qedhere} \]
\end{proof}

\section{A brief glimpse of infinity} \label{sec:glimpse}%
We now take the time to discuss some aspects of \oo\-category theory for readers
who mostly or exclusively work in a $1$-categorical setting.
In this survey section, we will not include any proofs, but we give lots of
pointers to the literature for further reading.
We hope that the reader who is not familiar with \oo\-category theory will
find enough insights as to get a comfortable feeling about the statements
and arguments of the other sections.

In \nameCrefs{sec:nutshell}~\ref{sec:nutshell} and~\ref{sec:workingin},
we give a rigorous definition of the term \emph{\oo\-category}
and explain how it relates to the concept of higher morphisms.
\cref{sec:importing} is dedicated to the question how interesting
\oo\-categories can be constructed from 1-categorical data,
and in \cref{sec:othernames} we note some caveats regrading the
nomenclature around the term \enquote{\oo\-category}.

We presume some familiarity with basic category theory,
e.g. as in the first two chapters of Mac Lane's book%
\icite[Sec.~I.1--4, I.8, II.1--4]{b:maclane98}. Furthermore, we note:
\begin{thConvention}
    Starting with this section, \emph{category} means \emph{1-category},
    and we will explicitly use \emph{\oo\-category} as in~\cref{def:infcat}.
\end{thConvention}

\subsection{From 1 to \pdfInfty{} in a nutshell} \label{sec:nutshell}%
Our first goal is to give a precise definition of \oo\-categories in terms of Lurie
and to observe that there's a formal way to turn 1-categories and functors into
\oo\-categorical ones. We refer the reader to the literature%
\icite{b:cisinski19}%
\mcite{pp:riehl20}%
\mcite{pp:groth15}%
\mcite{web:kerodon}%
\mcite{b:HTT}
for a more comprehensive treatment of the subject.

The underlying foundational object for this theory is the \emph{simplicial
set}, a generalization of (ordered) simplicial complexes, for which we give
a rigorous definition in the following but almost no illustration or background.
We recommend the beautiful survey article by Friedman\icite{p:friedman12} as an
introduction to the subject and the textbook by Goerss and
Jardine\icite{b:goerssjardine09} for further reading.

\begin{thDef}[simplicial set] \label{def:simplicialset}%
    \begin{enumerate}
        \item \label{def:simplicialset:preordercat}%
            Let $(X,\leq)$ be a preorder, i.e. $\leq$ is a transitive and
            reflexive relation on~$X$. We define the \emph{preorder category
            $\Pre(X,\leq)$} to have $X$ as objects and a unique morphism from
            $x$ to~$x'$ exactly if $x \leq x'$ (and no further morphisms).
        \item
            For $n\in\N$ let $\ordcat n \defeq \Pre(\setZeroto n, \leq)$.
            Let $\simplexcat$ be the full subcategory of~$\Cat$ on the objects
            $\{\ordcat n \Mid n\in\N\}$.
        \item\label{def:simplicialset:category}\label{def:simplicialset:indexnotation}%
            A \emph{simplicial set} is a functor $\simplexcat\op\to\Set$. The
            \emph{category~$\sSet$ of simplicial sets} is the functor category
            $\Fun(\simplexcat\op,\Set)$. Hence, a morphism of simplicial sets is
            a natural transformation of functors. For a simplicial set~$X$
            and $n\in\N$ one often uses the notation~$X_n \defeq X(\ordcat n)$.
        \item
            Let $X$ be a simplicial set. Then a \emph{simplicial subset} of~$X$
            is a simplicial set~$U$ such that for all $n\in\N$ we have
            $U_n\subset X_n$ and for all morphisms $f\colon [m]\to\ordcat n$
            in~$\simplexcat$ we have $U(f) = X(f)\restrict{U_n}$.
    \end{enumerate}
\end{thDef}

\begin{thRemark}[smaller simplicial sets]
    As a variant of~\crefitem{def:simplicialset}{:category}, we could
    replace $\Set$, the category of large sets, by the category of small
    sets. The resulting simplicial sets would then be suitable for
    modelling objects of~$\Top$ as in simplicial homotopy
    theory\icite{b:goerssjardine09}, but would be too small
    for our purposes.
\end{thRemark}

\begin{thRemark}[sets as simplicial sets]
    By an extension of a set~$U$ to a simplicial set, we mean a simplicial
    set~$X$ with $X_0 = U$.
    Such an extension always exists by virtue of the constant functor
    at~$U$.
\end{thRemark}

\begin{thRemark}[operations with simplicial sets] \label{rem:SmallestSimplicialSet}%
    Many operations one normally does with plain sets can be lifted to simplicial
    sets by just applying them \enquote{dimensionwise}. For instance, for two
    simplicial subsets $U,V$ of the same simplicial set, one can define
    $U \intersect V$ by $(U\intersect V)_n \defeq U_n \intersect V_n$ for all
    $n\in\N$ and restriction on morphisms. In particular the expression
    \enquote{smallest simplicial subset with some property} can be made precise
    by taking intersections as usual.
    
    The formal reason why this works, is the fact that limits and colimits in a
    functor category are just computed \enquote{pointwise} if
    they exist in the target category. Taking intersection is a pullback
    of sets, for example.
\end{thRemark}
With \cref{rem:SmallestSimplicialSet} in mind, we can now introduce some
fundamental simplicial sets:

\begin{thExamples}[standard simplex, horn] \label[thExample]{ex:simplicialsets}%
\hfill%
\begin{enumerate}[topsep={2.5pt plus 2pt minus 1pt}]
\item\label{ex:simplicialsets:cofacemap}%
    Let $n\in\N$.
    The simplicial set $\stdsimplex n \defeq \Mor_{\simplexcat}(\blank,\ordcat n)$
    is the \emph{standard $n$\-simplex}.
    For $n\geq1$ and $i\in\setZeroto n$ the unique functor
    $\coface{n-1}{i}\in\Mor_{\simplexcat}(\ordcat{n-1},\ordcat{n}) = (\stdsimplex n)_{n-1}$ with
    $\coface{n-1}{i}(\setZeroto{n-1}) = \setZeroto n\setminus\{i\}$
    is the \emph{$i$-th coface map}.
\item
    For $n\in\N$
    the smallest simplicial subset of~$\stdsimplex{n+1}$ that contains
    $\{\coface{n}{0},\dots,\coface{n}{n+1}\}\subset(\stdsimplex{n+1})_n$
    is the \emph{simplicial $n$\-sphere}.
\item
    Let $n\in\N[\geq1]$ and $i\in\setZeroto n$. The
    \emph{$(n,i)$\-horn~$\horn n i$} is the smallest simplicial subset of the
    standard $n$\-simplex that contains $\{\coface{n-1}{0},\dots,\coface{n-1}{n}\}
    \setminus\{\coface{n-1}{i}\} \subset (\stdsimplex n)_{n-1}$.
\end{enumerate}
\end{thExamples}
Informally speaking, the simplicial $(n-1)$\-sphere is obtained from
$\stdsimplex n$ by removing the \enquote{interior} of the $n$\-simplex;
and by further removing the $(n-1)$\-dimensional face opposite to the
$i$-th vertex we get the horn~$\horn n i$. The latter now plays the
main role in the definition of an \oo\-category:
\begin{thDef}[\pdfInfty-category] \label{def:infcat}%
    An \emph{\oo\-category} is a simplicial set~$C$ with the following
    property: For all $n\in\N[\geq1]$ and all $i\in\{1,\dots,n-1\}$ every
    morphism $\horn n i\to C$ factors over the inclusion
    $\horn n i\incl\stdsimplex n$.
    \[
        \begin{tikzcd}
            \horn n i \ar[r] \ar[d, inclusion, xshift=-2.5pt] & C \\
            \stdsimplex n \ar[ru, dashed, "\exists"']
        \envEndsHere
        \end{tikzcd}
    \]
\end{thDef}
A usual formulation of this property is the following:
In an \oo\-category all so-called \emph{inner horns} ($0<i<n$) have
a \emph{not} necessarily unique(!) \enquote{filler}.
Those fillers are witnesses of (higher dimensional) compositions in an \oo\-category
and the existence statement then ensures that a composition can always be found.
See \cref{rem:compositions} for the case of~$\horn 2 1$; for examples using
higher dimensional horns and the existence of their fillers, we refer to the
literature%
\icite[Sec.~1.2.3]{b:HTT}%
\mcite[proof of Lem.~1.6.2]{b:cisinski19}%
\mcitekerodonX{\kerodontag{003U}, \kerodontag{0041}}.
For nomenclature linking the different pieces of an \oo\-category to classical
categorical terms, see \cref{def:nomenclature}.
\begin{thExample}
    For all $n\in\N$, the standard $n$\-simplex is an \oo\-category. On the other hand,
    for $n\in\N[\geq1]$, the simplicial $n$\-sphere is \emph{not} an \oo\-category.
\end{thExample}

\begin{thRemark}[functor \pdfInfty-categories] \label{rem:internalhom}%
    For \oo\-categories $C$ and~$D$, a functor $C\to D$ is simply taken to be a
    morphism of simplicial sets.
    
    A certain construction\icitekerodon{0060}\mcite[Sec.~I.5]{b:goerssjardine09}
    turns the category of simplicial sets into a \emph{cartesian closed}
    category\icite[Sec.~IV.6]{b:maclane98}: it provides a simplicial set~$Y^X$
    that extends the set of morphisms between two simplicial sets $X$ and~$Y$,
    i.e. with $(Y^X)_0 = \Mor_{\sSet}(X,Y)$.  This is often called an
    \emph{internal Hom}, because it is again an object of the ambient category,
    or in this specific case the \emph{function complex from $X$ to~$Y$}.
    It can be shown\icite[Cor.~3.2.10]{b:cisinski19} that $Y^X$ is again an
    \oo\-category whenever $Y$~is an \oo\-category.
    In particular, this shows how functors $C\to D$ for fixed \oo\-categories
    $C$ and~$D$ give rise to a functor \oo\-category $\Fun(C,D)$.%
    \envEnds
\end{thRemark}

\begin{thRemark}[size issues] \label{rem:sizeissues}%
    As the reader has probably already noticed, the operation of taking
    \enquote{all functors $C\to D$} might or might not yield a well\-defined
    object, depending on the set theory that one chooses as a mathematical
    foundation. To be more specific: If one considers functors between categories
    \enquote{of the same size} (in terms of the \enquote{number} of objects and
    morphisms) one often obtains something that is \enquote{bigger} than that.
    On the other hand, if one imposes the correct size constraints on $C$
    and~$D$, the result can also be kept \enquote{small enough}.
    
    For example, suppose that we work within NBG (von Neumann, Bernays, Gödel)
    set theory. Then a category is usually defined to allow for a proper
    class of objects, but most often only morphism \emph{sets} between any two
    objects are permitted. If we consider such categories, it is well\-known that
    one is usually not permitted to form the functor category between two
    categories whose classes of objects are both proper classes. Though, if we
    consider a \emph{small} category~$C$, i.e. its class of objects is a set,
    and another category~$D$, the functor category $\Fun(C,D)$ is again a
    well\-defined category in the sense of this paragraph.
    
    Analogously, such considerations can be made for \oo\-categories when
    appropriate size restrictions are in place\icite[Cor.~5.7.7]{b:cisinski19}.
    Even then, however, the \enquote{two layers} of NBG set theory are not
    enough for the arguments in \cref{sec:approxchaincomplexfunctor}:
    there we want to look at the functor category $\Fun(\Top_*,\ChR)$
    where the size conditions are \emph{not} met. A usual way to circumvent this
    problem is the usage of so-called \emph{Grothendieck universes}. Roughly
    speaking, this means the following: A set $U$ is called a \emph{universe}
    if it satisfies certain closure properties like the \emph{power set axiom}:
    \[ \forall x\in U, \quad \{ y \Mid y \subset x \} \in U  .\]
    Then an additional axiom, the \emph{universe axiom}, is imposed:
    \[ \forall x \quad \exists\, U \text{\,universe}, \quad x\in U  .\]
    One can then just choose a universe that is big enough to support all the
    operations one wants to make, i.e. such that \enquote{enough layers
    exist}. As one may convince oneself, for the formulation of the statements
    in this article, it is actually enough to consider just \enquote{one more
    layer} on top of NBG, giving it a name like \enquote{superclasses} or
    \enquote{conglomerates} if one wants to maintain the hierarchy that is
    already in place.
    
    For further information, we refer the reader to the literature about size
    issues and possible solutions:
    A concise introduction to Grothendieck universes can be found in a paper by
    Low\icite{pp:low13};
    a gentle explanation of the \enquote{conglomerate solution} can be found in
    the category theory book by Ad\'amek, Herrlich and
    Strecker\icite[2~Foundations]{b:adamekherrlichstrecker90};
    an expository article by Shulman compares set theoretical foundations for
    category theory from a modern perspective\icite{pp:shulman08}.%
    \envEnds
\end{thRemark}
\begin{thRemark}[size levels of our objects]
    We say a category is of a specific size, if its set of objects and
    set of morphisms are both of that size.
    Under the explicit choices that we made in the introduction
    (page~\pageref{sec:conventions}), we have, among others:
    \begin{itemize}
    \item
        The sets~$\N,\Z,\R$ and all underlying sets of objects of~$\Top$
        are small, and $\simplexcat$~is a small category.
    \item
        The categories $\Top$, $\Ch_k$ and the category of small sets are
        large categories. However, they are locally small, i.e. the
        morphism set between every two objects is small.
    \item
        The categories $\Set$, $\Cat$ and~$\sSet$ are very large, but
        locally large.%
    \item
        The category of very large categories is huge.%
        \envEndsHere
    \end{itemize}
\end{thRemark}

Now that we have defined what an \oo\-category is, we would certainly want to
\enquote{import} all the 1-categories that we already know of into this setting.
Put differently, we really want that \oo\-categories provide a
generalization of 1-categories. This can be done via the \emph{nerve} functor:
\begin{thDef}[nerve of a 1-category] \label{def:1nerve}%
    Let $\iota\colon\simplexcat\incl\Cat$ denote the inclusion functor. Then the
    \emph{nerve functor $\nerve\colon\Cat\to\sSet$} arises from the functor
    \[ \Mor_{\Cat}(\iota\op\blank,\blank)\colon\simplexcat\op\cross\Cat\to\Set \]
    by \emph{currying}, i.e. $\nerve(X) = \Mor_{\Cat}(\iota\op\blank,X)$.%
    \envEnds%
\end{thDef}
More formally, this currying construction can be seen as an application of
the exponential law in the category of very large categories:
$\Fun(C\cross C', D) \isom \Fun(C',\Fun(C,D))$, analogously to the
exponential law for functions\icite[Exerc.~2 of Sec.~II.5]{b:maclane98}.
\begin{thExample}[standard simplex via nerve]
    For all $n\in\N$ we have $\stdsimplex n = \nerve(\ordcat n)$.
\end{thExample}
\begin{thProposition}[nerve is \pdfInfty-category%
    \titlecitesep%
    \titlecite{b:cisinski19}{Prop.~1.4.11, Ex.~1.5.3}%
    \titlecite{b:HTT}{Prop.~1.1.2.2}%
    \titlecitekerodon{002Z}%
    ]
    For every category~$C$, the nerve $\nerve(C)$ is an \oo\-category. Furthermore, the
    nerve functor~$\nerve$ is fully faithful, i.e. for all categories $C$ and~$D$
    the map $\Mor_{\Cat}(C,D) \ni F \mapsto \nerve(F) \in
    \Mor_{\sSet}(\nerve(C),\nerve(D))$
    is bijective.
\end{thProposition}
This means that whenever one embeds two 1-categories into the \oo\-world via
the nerve, there will be exactly \enquote{the same} functors between them as
there were before. So in this sense, \oo\-categories generalize
1-categories.

However, while it is nice that we have 1-categories as \oo\-categories now,
it's not an enhancement of the original 1-categories. Often, one
builds \oo\-categories by other means and then they contain some
homotopy theoretic information of the original 1-category.
In \cref{sec:importing} we will go into a few more details.

\subsection{Working in \pdfInfty-categories} \label{sec:workingin}%
Now that we know how to model \oo\-categories by simplicial sets, we can
set up some nomenclature to make \emph{working} in an \oo\-category more similar to
working in a 1-category.

\begin{thDef}[nomenclature for \pdfInfty-categories] \label{def:nomenclature}%
    Let $C$ be an \oo\-category and let us use the notations of
    \crefitem{def:simplicialset}{:indexnotation} and
    \crefitem{ex:simplicialsets}{:cofacemap}. Then
    \begin{itemize}
        \item \emph{objects of~$C$} are the elements of~$C_0$,
        \item \emph{morphisms of~$C$} are the elements of~$C_1$,
        \item \emph{2-morphisms of~$C$} are the elements of~$C_2$,
        \item and generally for $n\in\N$ the \emph{$n$\-morphisms of~$C$}
                are the elements of~$C_n$.
    \end{itemize}
    For two objects $x$ and~$y$ in~$C$, a morphism $x\to y$ is a morphism~$f$
    of~$C$ such that $C(\coface{0}{1})(f) = x$ and $C(\coface{0}{0})(f) = y$.
    For an object~$x$ of~$C$ the morphism $C(\ordcat 1\to\ordcat 0)(x)$ of~$C$
    is the \emph{identity morphism of~$x$}, where we simply write
    $\ordcat 1\to\ordcat 0$ for the unique such functor.%
    \envEnds%
\end{thDef}
Note that \enquote{morphism} and \enquote{1-morphism} mean exactly the same
thing, as do \enquote{object} and \enquote{0-morphism}, though the latter
is not used by all authors.
With this notation at hand, one can, at least superficially, start to use
an \oo\-category just as a 1-category. However, there is one big catch:
\begin{thRemark}[lack of composition map] \label{rem:compositions}%
    In an \oo\-category~$C$, there is usually no \enquote{composition map
    $\Mor(y,z)\cross\Mor(x,y)\to\Mor(x,z)$} like in 1-categories. Instead, given
    morphisms $f\colon x\to y$ and $g\colon y\to z$ there could be many possible
    compositions of $f$ and~$g$ -- and this is indeed one of the key notions in
    the generalization from 1-categories to \oo\-categories. More precisely,
    $f$ and~$g$ give rise to a map $\horn 2 1\to C$ of simplicial sets,
    and any extension $\sigma\colon\stdsimplex 2\to C$ to the standard simplex
    \[
        \tikzcdset{diagrams={cramped, column sep={1.2cm,between origins}, row sep=0.7cm}}
        \begin{tikzcd}
            & y \ar[dr, "g"] & \\
            x \ar[ur, "f"] & & z
        \end{tikzcd}
        \surround\qquad\rightsquigarrow
        \begin{tikzcd}
            & y \ar[dr, "g"] \ar[d, phantom, "{\scriptstyle|\,|\,|\,|}" description] & \\
            x \ar[ur, "f"] \ar[rr, dashed, "h"'] & {} & z
        \end{tikzcd}
    \]
    determines \emph{a} composition~$h$ of $f$ and~$g$, namely:
    \[ h = C(\coface{1}{1})\bigl( \sigma_{\ordcat 2}(\Id_{\ordcat 2}) \bigr)
        \,\in C_1
    . \]
    The existence of such an extension is guaranteed by the fact that $C$ is an
    \oo\-category, but there can (and usually will) be many! In the situation above,
    we will say that \emph{$\sigma$~witnesses that $h$ is a composition of $f$
    and~$g$}.
    
    A useful analogy is the following: When defining the concatenation of two
    compatible paths with common domain interval in a topological space,
    one has to subdivide the interval, which is usually done by splitting it in
    the middle. This might look most symmetric at first, but as soon
    as we consider three paths, in fact \emph{no} such subdivision will make the
    concatenation map associative. One way to enforce this, is passing to the
    homotopy category, but as mentioned below \pcref{rem:homotopycategories},
    this might be a bad idea. Instead, we could just declare that \emph{any}
    concatenation (regardless of the chosen interval subdivision) is a valid
    composition of the paths, i.e. we allow many different compositions of two
    paths.%
    \envEnds
\end{thRemark}
With this in mind, the next natural question is: How does the notion of
\emph{(commutative) diagrams}\icitekerodon{005H} enter \oo\-category theory?
Indeed, since we already know what a functor of \oo\-categories is,
we could say, analogously to the usual definition for 1-categories, that
a diagram in an \oo\-category~$C$ is a functor $C'\to C$ from some \oo\-category~$C'$
(the \enquote{shape of the diagram}) to~$C$.
Though this would suffice on a purely technical level%
\icite[Thm.~7.3.22]{b:catLR}\mcite[Prop.~4.2.3.14]{b:HTT}
it is more convenient to allow \enquote{shapes} that are not \oo\-categories
themselves:
\begin{thDef}[diagram in an \pdfInfty-category]
    Let $C$ be an \oo\-category and let $K$ be a simplicial set. A \emph{diagram
    in~$C$ indexed by~$K$} is a morphism of simplicial sets $K\to C$.
\end{thDef}
\begin{thExample}[diagrams indexed by simplices] \label{ex:yoneda}%
    Let $n\in\N$ and let $C$ be an \oo\-category. A diagram $\stdsimplex n\to C$ is
    equivalently an $n$\-morphism of~$C$; more precisely: the map
    \[ \Mor_{\sSet}(\stdsimplex n, C) \to C_n, \quad
        \sigma\mapsto \sigma_{\ordcat n}( \Id_{\ordcat n} )
    \]
    is a bijection. (This is an instance of the Yoneda
    lemma\icite[Sec.~III.2]{b:maclane98}.)
    It is common to identify $n$\-morphisms of~$C$ with diagrams
    $\stdsimplex n\to C$ via this bijection whenever it is convenient.
\end{thExample}
\begin{thExample}[commutative squares] \label{ex:commutativesquares}%
    Let $K \defeq \nerve(\ordcat 1\cross\ordcat 1)$ be the nerve of the product
    category $\ordcat 1\cross\ordcat 1$ (taken in~$\Cat$). More explicitly,
    the latter can be depicted as
    \[ \begin{tikzcd}
        (0,0) \ar[r, "f"]\ar[d, "f'"']\ar[rd, "h"] & (1,0) \ar[d, "g"] \\
        (0,1) \ar[r, "g'"'] & (1,1)
        \end{tikzcd}
    \]
    and its nerve has the obvious additional 2-simplices~$\sigma$ and~$\tau$
    \enquote{filling the two triangles}. Then a diagram $p\colon K\to C$
    amounts to the data of
    \begin{itemize}
        \item four objects $p( (0,0) )$, $p( (1,0) )$, $p( (0,1) )$ and
            $p( (1,1) )$,
        \item five morphisms $p(f)$, $p(g)$, $p(f')$, $p(g')$ and $p(h)$, and
        \item two 2-morphisms $p(\sigma)$ and $p(\tau)$
    \end{itemize}
    of~$C$, such that
    \begin{itemize}
        \item the morphisms have the correct source and target and
        \item the 2-morphisms $p(\sigma)$ and~$p(\tau)$ witness that $p(h)$ is a
            composition of $p(f)$ and~$p(g)$ as well as of $p(f')$ and~$p(g')$.
            (Here we apply \cref{ex:yoneda}.)
    \end{itemize}
    As an illustration, we consider the case that $C$~is an \oo\-category of
    topological spaces, e.g. as in \cref{rem:hcnerve} via the homotopy
    coherent nerve or as the localization of~$\Top$ at weak homotopy
    equivalences in the sense of~\cref{rem:localization}.
    Then in a diagram as above, $p(\sigma)$ and~$p(\tau)$ are homotopies
    $p(h) \equiv p(g)\after p(f)$ and $p(h) \equiv p(g')\after p(f')$.
    So in particular, every such diagram in~$C$ gives a 1-categorical commutative
    diagram in the homotopy category of topological spaces. But it is crucial to understand
    that the latter notion is just a \emph{property} of a 1-categorical diagram
    while an actual diagram~$K\to C$ carries the 2-morphisms $p(\sigma)$
    and~$p(\tau)$ as additional \emph{data}.%
    \envEnds%
\end{thExample}
\begin{thRemark}[homotopy coherence] \label{rem:htpycoherence}%
    The last example already shows that commutative diagrams in
    \oo\-categories are quite a bit more delicate to handle than in
    1-categories. Of course, more complex diagrams than commutative squares
    require even more data to be specified by a diagram. In particular, diagrams
    indexed by \enquote{infinite shapes}, e.g. $\nerve(\Pre(\N,\leq))$,
    may require chosen $n$\-morphisms in the target \oo\-category for
    \emph{all}~$n\in\N$ in a compatible way -- one also speaks of a
    \emph{coherent} choice of higher morphisms.
    For further reading about homotopy coherence, we recommend notes by
    Riehl\icite{pp:riehl18}, which also contain an
    example\icite[Ex.~I.3.4]{pp:riehl18} of a homotopy commutative diagram that
    cannot be made homotopy coherent.
\end{thRemark}

\begin{thRemark}[sloppy notation] \label{rem:sloppynotation}%
    As we have seen in \cref{ex:commutativesquares} (we will reuse its notation
    here), a commutative square in an \oo\-category needs to specify a
    \enquote{diagonal} morphism.
    Indeed, if one defines a commutative square in a 1-category~$C'$
    analogously as a functor $F\colon [1]\cross[1] \to C'$, one a priori has
    to define $F(h)$. But since compositions in 1-categories are
    unique, one easily sees that keeping this piece of data is \emph{redundant},
    i.e. inferable from $(F(f),F(g))$. So for such a 1-categorical commutative
    square, the restriction of~$F$ to the \enquote{outer} four morphisms
    uniquely determines~$F$ and this is why we usually don't mention the
    \enquote{diagonal} at all.
    
    However, this is not true anymore for diagrams in \oo\-categories! So
    when speaking of a commutative square
    \[ \begin{tikzcd} a \ar[r] \ar[d] & b \ar[d] \\ c \ar[r] & d \end{tikzcd} \]
    in an \oo\-category~$C$, one really means a diagram
    $p\colon K\to C$ where $p( (0,0) ) = a\to d = p( (1,1) )$ as well as
    $p(\sigma)$ and~$p(\tau)$ are not visible, although they are implicit data
    associated to the square.%
    \envEnds
\end{thRemark}

\begin{thOutlook}[(co)limits] \label{out:limits}%
    Having said what a diagram is, one would like to speak about limits and
    colimits of such. Unfortunately, we cannot fully define this notion here,
    because it requires a more technical insight into the theory of
    \oo\-categories and simplicial sets. To give an idea of why this is the
    case, we may try to transplant the usual picture of how one thinks about
    limits in 1-categories to the \oo\-categorical world: traditionally, one
    starts with the notion of a \emph{cone} on a diagram. While this is
    easy for 1-categorical, i.e. \enquote{1-dimensional} diagrams, one already
    has to reflect on the issue, how to define a cone on an
    \oo\-categorical, i.e \enquote{higher dimensional} simplicial set shaped
    diagram. However, it is indeed not difficult to find explicit
    formulae that define the cone of a simplicial set%
    \icitekerodonX{\kerodontag{0172}, \kerodontag{0177}},
    which in turn may be used to formulate what a cone on an \oo\-diagram is.
    
    Next, we would like to single out \emph{universal} cones. In 1-category
    theory, those are then called limits or limit cones of the original
    diagram, and the universal property roughly says that \emph{every} cone
    factors through such a limit cone. The latter property, though, cannot
    immediately be transferred to \oo\-cones for several reasons: first,
    remember that we do not have unique compositions of morphisms
    \pcref{rem:compositions}, and second, one would have to take higher
    simplices into account.
    
    To solve this problem, one arranges for the \oo\-cones on a given diagram
    to be the objects of a certain \oo\-category. The latter then includes all the
    information about the higher morphisms and certain objects of this \oo\-category
    will then be called universal \oo\-cones.
    
    All of this applies to cones, yielding \emph{limits}, as well as dually to
    cocones, yielding \emph{colimits}. In the first case, the cone point is the
    initial vertex of each simplex, in the second case the final vertex.
    For an expository account of the matter, skipping most technicalities yet
    providing all necessary intermediate steps, see Groth's
    notes\icite[Sec.~2]{pp:groth15}. All details can be found in the
    literature%
    \icite[Sec.~6.2]{b:cisinski19}%
    \mcite[Sec.~1.2.13]{b:HTT}.%
    \envEnds
\end{thOutlook}

\begin{thExample}[pullbacks and pushouts]
    As an informal example of the previous outlook, we say a bit more about
    pullback and pushout squares. For this, we consider diagrams indexed by
    the horns $\horn 2 2$ and~$\horn 2 0$. The cone on $\horn 2 2$ and
    the cocone on $\horn 2 0$ are both isomorphic to $\nerve([1]\cross[1])$;
    a fact that is intuitively clear by the following pictures:
    \[
        \tikzcdset{diagrams={cramped}}
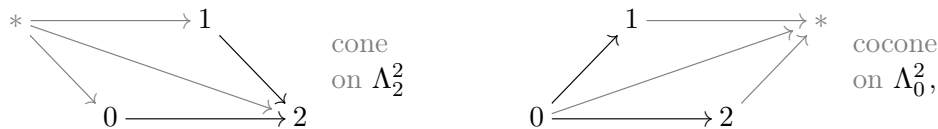
        \begin{tikzcd}
            {\color{gray}\ast} \ar[rr, gray] \ar[dr, gray]
                \ar[ drrr, gray
                   , start anchor=base east
                   , shorten >=5pt, end anchor={[yshift=1pt]} ]
            & & 1 \ar[dr] &
            \\
            & 0 \ar[rr] & & 2
        \end{tikzcd}
        \quad
        \parbox{1cm}{{\color{gray}cone\\ on} $\horn 2 2$}
        \hspace{1.5cm}
        \begin{tikzcd}
            & 1 & &
            {\color{gray}\ast} \ar[ll, gray, <-] \ar[dl, gray, <-]
                \ar[ dlll, gray, <-
                   , end anchor={[yshift=1pt]}
                   , shorten <=5pt, start anchor={[yshift=-1pt]} ]
            \\
            0 \ar[rr] \ar[ur] & & 2
        \end{tikzcd}
        \quad
        \parbox{1.3cm}{{\color{gray}cocone on} $\horn 2 0$\rlap{\kern1pt,}}
    \]
    i.e. both of them may index commutative squares
    \pcref{ex:commutativesquares} in an \oo\-category.
    Given a commutative square $p\colon\nerve([1]\cross[1])\to C$, we say that
    $p$ is
    \begin{itemize}
        \item a \emph{pullback square} if the cone on~$\horn 2 2$ that it
            determines is universal and it is
        \item a \emph{pushout square} if the cocone on~$\horn 2 0$ that it
            determines is universal,
    \end{itemize}
    both in the sense of \cref{out:limits}.
\end{thExample}

\begin{thOutlook}[stable \pdfInfty-category] \label{out:stable}%
    The property of being a stable \oo\-category\icite[Def.~1.1.1.9]{b:HA},
    which is especially important in the context of functor
    calculus\icite[Ch.~6]{b:HA}, can be defined in terms of pullbacks and
    pushouts\icite[Prop.~1.1.3.4]{b:HA}:
    An \oo\-category is \emph{stable} if
    \begin{itemize}
        \item
            it has a zero object, i.e. an object that is both initial and
            terminal,
        \item
            it has all finite limits and colimits, and
        \item
            a square in it is a pushout if and only if it is a pullback.
    \end{itemize}
    This concept originates from stable homotopy theory, which Lurie
    recalls and discusses in depth in the \oo\-categorical
    setting\icite[Sec.~1.4]{b:HA}.
    
    It should also be noted, that a precursor is the notion of a
    \emph{stable model category}\icite[Ch.~7]{b:hovey07}, which can
    similiarly be defined as a model category that has a zero object
    and in which homotopy pushouts and homotopy pullbacks
    coincide\icite[Rem.~7.1.12]{b:hovey07}.
    (For a primer on homotopy pushouts in~$\Top$, see
    \cref{sec:modelcatbits}.)%
    \envEnds
\end{thOutlook}

\subsection{How to produce \pdfInfty-categories} \label{sec:importing}%
The following method is, historically and because of its role in Lurie's
book\icite{b:HTT}, one of the most important constructions that convert
1-categorical input into a corresponding \oo\-category, taking the
homotopical structure into account:
\begin{thRemark}[homotopy coherent nerve] \label{rem:hcnerve}%
    Let $C$ be a \emph{simplicially enriched category}: instead of just morphism
    sets, each $\Mor_C(X,Y)$ is a simplicial set and the usual requirements
    regarding composition are applied mutatis mutandis\icitekerodon{00JQ}.
    For such a category~$C$, there is a construction analogous to the nerve of a
    1-category \pcref{def:1nerve}, called the \emph{homotopy coherent nerve} or
    \emph{simplicial nerve}, that builds a simplicial set $\nervehc(C)$ out of~$C$%
    \icite[Def.~1.1.5.5]{b:HTT}%
    \mcite[Sec.~2.4]{p:duggerspivak11a}.
    Under suitable conditions, $\nervehc(C)$ is an \oo\-category%
    \icite[Prop.~1.1.5.10]{b:HTT},
    which is in some sense a \enquote{homotopical thickening} of the
    nerve~$\nerve(C)$ of~$C$.
    
    To illustrate the last comment a bit more, we consider the example~$C=\Top$:
    For spaces $X$ and~$Y$, the set of continuous maps~$X\to Y$ can be extended
    to a simplicial set that satisfies
    \[ \Mor_{\Top}(X,Y)_n
        = \{ \text{continuous maps } \topstdsimplex{n}\cross X \to Y \}
    \]
    for all $n\in\N$\icitekerodon{00JV}; here, $\topstdsimplex{n}$ denotes the
    standard topological $n$-simplex.
    By easy inspection, one sees that the 2-simplices~$\nerve(\Top)_2$ of the usual
    nerve correspond to commuting triangles in~$\Top$, i.e. to triples $(f,g,h)$ of
    morphisms in~$\Top$ such that $g\after f = h$.
    By contrast, $\nervehc(\Top)_2$~will correspond to triangles commuting up to
    homotopy together with such a homotopy\icitekerodon{00KX}, i.e. to
    quadruples~$(f,g,h,H)$ where $f,g,h$ are morphisms in~$\Top$ and $H$~is a
    homotopy from $g\after f$ to~$h$ in the sense of simplicially enriched
    categories\icitekerodon{00JW}. In our example, the latter notion happens
    to coincide with the usual notion of homotopies
    in~$\Top$\icitekerodon{00JX}.
    The upshot is, that we still have the strictly commuting triangles, but also
    a lot of other ones, which allows for more flexibility.
\end{thRemark}
In certain situations, the homotopy coherent nerve can provide
\oo\-categorical versions of homotopy categories:
\begin{thRemark}[homotopy categories] \label{rem:homotopycategories}%
    Let $C$ be a \emph{simplicial model category}, i.e. a simplicially enriched
    category that is also a \emph{model category} in a compatible way. The
    structure of a model category\icite[Def.~2.2.1]{b:cisinski19} that is
    subject to some conditions, of course, is the following:
    There is a distinguished class of morphisms~$W$ of~$C$, called
    \emph{weak equivalences}, that one is interested in from a
    homotopical point of view. Furthermore there are two distinguished
    classes of auxiliary morphisms, called \emph{fibrations} and
    \emph{cofibrations}.
    
    There is a well\-known construction that assigns to a model category~$C$ with
    weak equivalences~$W$ its \emph{homotopy category}, which in turn presents
    itself as a model for the localization $C[W\inv]$,
    i.e.\icite[Def.~2.2.8]{b:cisinski19} the universal category obtained
    from~$C$ by making all morphisms in~$W$ isomorphisms. For example, looking
    at the category $\Top$ of topological spaces with $W$ the class of homotopy
    equivalences, the corresponding homotopy category can be obtained from
    $\Top$ by modding out the \enquote{is homotopic to} equivalence relation on
    all morphism sets.
    
    While 1-categorical homotopy categories are certainly useful, working with
    them tends to lose information in the sense that the homotopies that make
    diagrams commutative are not part of the data. They are merely required to
    exist for any choice of two parallel morphisms, but not necessarily
    in a coherent way, see also \cref{rem:htpycoherence}.
    
    In the case of the simplicial model category~$C$, this is remedied by
    an \oo\-categorical enhancement of~$C[W\inv]$: One takes a
    certain full subcategory $C^\circ$ of~$C$ (which inherits a simplicial
    enrichment from~$C$) and forms $\nervehc(C^\circ)$. This gives an
    \oo\-category that is tightly linked to the homotopy category of~$C$. (More
    specifically: one can also define what the homotopy category of an \oo\-category
    is and if one applies this notion to $\nervehc(C^\circ)$ the result will be a
    1-category that is equivalent to $C[W\inv]$.
    This can be seen by combining several facts about homotopy categories%
    \icitekerodon{00M4}%
    \mcite[Prop.~9.5.24\,(2)]{b:hirschhorn09}%
    \mcite[Sec.~7.5.6]{b:hirschhorn09}%
    .)
\end{thRemark}

However, while these constructions are certainly useful and applicable in some
important cases, they require a lot of structure -- which, in applications, is
often not naturally available. For example, one might want to consider (model)
categories where no extension to a \emph{simplicial} model category exists or is
known. And even if there is, the functors one would like to consider can most
likely not easily be lifted to respect the simplicial enrichment. Thus, it is
useful to have alternative approaches, specifically ones that allow to import a
given 1-categorical situation, including functors, more directly.
The following notion captures the minimal setting that is necessary to talk
about homotopical situations:
\begin{thDef}[relative category] \label{def:relcat}%
    A \emph{relative category $(C,W)$} consists of
    \begin{itemize}
        \item
            a category~$C$, together with
        \item
            a class of morphisms~$W$ of~$C$, called \emph{weak equivalences}.
    \end{itemize}
    A \emph{relative functor}, also called a \emph{homotopical functor},
    from a relative category~$(C_1,W_1)$ to a relative category~$(C_2,W_2)$
    is a functor $F\colon C_1\to C_2$ that maps weak equivalences to weak
    equivalences, i.e. with $F(W_1) \subset W_2$.%
    \envEnds
\end{thDef}
Barwick and Kan have shown\icite{p:barwickkan12} that relative categories
can be used to model the homotopy theory of \oo\-categories (see also
\cref{rem:alternativeInftyCat}). Also note, that every model category
yields an example of a relative category by simply forgetting the
fibrations and cofibrations.

Similarly to \cref{rem:homotopycategories}, one is usually not only
interested in a 1-categorical localization~$C[W\inv]$ of a relative
category~$(C,W)$, but rather in a better behaved \oo\-categorical version
thereof. Classically, this was done by a process called \emph{simplicial
localization}, which was introduced and studied in a series of papers by
Dwyer and Kan\icite{p:dwyerkan80a,p:dwyerkan80b,p:dwyerkan80c}.
From this, one gets a simplicially enriched category -- but not of the form
to which \cref{rem:hcnerve} applies directly.
Instead of trying to rectify this, one can go straight to
localizations of \oo\-categories\icite[Sec.~7.1]{b:cisinski19}:
\begin{thRemark}[localization] \label{rem:localization}%
    Let $(C,W')$ be a relative category \pcref{def:relcat}.
    Using the terminology of \cref{def:nomenclature}, we can view
    $W'$~as a subset of the morphisms of either $C$~itself or of its
    nerve~$\nerve(C)$. Let $W$ denote the smallest simplicial subset
    of~$\nerve(C)$ that contains~$W'$.
    One can then form the localization~$W\inv\!\nerve(C)$, which
    is an \oo\-category right away\icite[Prop.~7.1.3]{b:cisinski19}.
    (It will have\icite[Rem.~7.1.6]{b:cisinski19} the same link to
    $C[W\inv]$ as described at the end of \cref{rem:homotopycategories}.)
\end{thRemark}

\subsection{Other names and other models} \label{sec:othernames}%
As the final act of this section, we should emphasize, that we
merely presented one facet of the theory of \oo\-categories. So
as not to confuse the novice reader, we purposely refrained from
mentioning this before; but since its complete concealment might
lead to frustration when consulting the literature, we deem it
necessary to inform the reader of the seemingly bizarre terminology
landscape:

\begin{thRemark}[other words for \pdfInfty-category]
    In the literature, there are three terms for the \emph{same} object,
    listed in reverse chronological order of appearance:
    \begin{itemize}
    \item
        \oo\-category \pcref{def:infcat}, as used by Lurie\icite{b:HTT},
    \item
        \emph{quasi-category}, Joyal's terminology%
        \icite{p:Joyal02}\mcite{pp:Joyal08}, and
    \item
        \emph{weak Kan complex}, introduced by Boardman and Vogt%
        \icite{b:BoardmanVogt73}\mcite{p:Vogt73}.
    \end{itemize}
    All of them are still commonly used, although it is customary to
    annotate the choice of \enquote{\oo\-category} with the phrase \enquote{in
    terms of Lurie} or similar, in order to avoid misunderstandings (see
    \cref{rem:alternativeInftyCat}).%
    \envEnds
\end{thRemark}

\begin{thRemark}[other objects that are called \pdfInfty-category]
    \label{rem:alternativeInftyCat}%
    On the other hand, there are \emph{different} objects that are also termed
    \enquote{\oo\-categories} in the literature:
    \begin{itemize}
    \item
        topologically enriched categories,
    \item
        simplicially enriched categories,
    \item
        Segal categories,
    \item
        complete Segal spaces,
    \item
        \ldots
    \end{itemize}
    This stems mainly from the fact, that during the evolution of higher
    category theory, the term \enquote{\oo\-category} was used to denote
    the somewhat vague concept of \enquote{category with infinitely many
    levels of morphisms}. To make this notion precise, several people developed
    different approaches, which resulted (among others) in the objects listed
    above. Nowadays, it is common to denote the conceptual idea, which they are
    all a model of, by the term \enquote{$(\infty,1)$\-category}. While a little
    longer than \enquote{\oo\-category}, this prevents ambiguities.%
    \envEnds
\end{thRemark}

For a more comprehensive account of the history and comparison between
the different conceptual aspects and models, we refer to the literature%
\icite{b:Bergner18}%
\mcite{p:Bergner10}%
\mcite[Introduction and Sec.~2--5]{pp:Joyal08}%
\mcite[Ch.~1]{b:HTT} %
and the \mbox{\emph{n}Lab}\icite{web:nlab}, starting at the articles
\enquote{higher category theory}, \enquote{(n,r)-category}, and
\enquote{quasi-category}.

\section{Excision for singular homology} \label{sec:singular}%
Let us consider singular homology (with arbitrary constant coefficients) on~$\Top$.
It satisfies classical excision, e.g. in the sense of having Mayer-Vietoris
sequences. As discussed before~\pcref{sec:approxchaincomplexfunctor},
this should rather be seen as a property of the underlying chain complex functor~$C$,
and indeed formal arguments then can be used to derive
Mayer-Vietoris sequences from this fact, see \cref{sec:classical}. In this
section we will show that $C$~really is excisive in the sense
of~\cref{def:excisive}, i.e. that the latter notion is a generalization of the
classical property.

For the entire section, we fix a ring~$k$ and a left $k$-module~$M$. Since
neither $k$ nor~$M$ play any role in the following, we simply denote by
$C \defeq C(\blank;M)$ the singular chain complex functor with
$M$-coefficients, and by $\Ch \defeq \Chk$ the category of chain complexes
over~$k$.

\begin{thTheorem}[\texorpdfstring{$C$}{C} is excisive] \label{thm:Cexcisive}%
    The \oo-categorical singular chain complex functor $C\colon\Top\to\Ch$ is excisive.
\end{thTheorem}
As this is often used as an example or motivation for the definition of an
excisive functor, it seems to be an easy fact for homotopy theorists, but an
explicit proof seems to be hard to find in the literature, so we will include
one here.

\subsection{Model category theoretical bits} \label{sec:modelcatbits}%
To prove the main theorem of this section, we will use the formalism of model
categories and we will show the following more explicit (1-categorical) version:
\begin{thProposition} \label{prop:Cexcisive1cat}%
    Homotopy pushout squares in~$\Top$ are sent to homotopy pullback squares
    in~$\Ch$ by the singular chain complex functor.
\end{thProposition}
Here, we consider the \emph{classical} or \emph{Quillen model
structure} on~$\Top$\icite[Sec.~2.4]{b:hovey07} and the
\emph{projective model structure} on~$\Ch$\icite[Sec.~2.3]{b:hovey07}.
Weak equivalences in the former are the weak homotopy equivalences, i.e.
continuous maps that induce isomorphism on all homotopy
groups\icite[Sec.~6.7]{b:tomdieck08}, and weak equivalence in the latter
are the quasi-isomorphisms, i.e. chain maps that induce an isomorphism on
homology.
No further familiarity with the model structures is required to
follow the rest of this and the overall logic of the subsequent sections.
However, in the proofs of the latter, we cannot fully avoid to mention
(co)fibrations, for example when we need to splice together
results from different references. The desperate reader may find an
introduction to (the nomenclature of) model categories in Cisinski's
book\icite[Sec.~2.2]{b:catLR}.

We will now give an ad~hoc definition of homotopy pushouts in~$\Top$, which can
be thought of as \enquote{thick pushouts}, as witnessed by the \emph{double
mapping cylinder} construction in the below definition. As we shall see
shortly, we will actually concern ourselves only with homotopy pushouts of a
certain form, but using that as our definition would hide the essence of this
homotopy theoretic notion entirely.
Thus we have\icite[Def.~15.5.8]{b:hirschhorn09}\mcite[Def.~3.7.1]{b:munsonvolic15}:
\begin{thDef}[homotopy pushout (square) in~\texorpdfstring{$\Top$}{Top}]
    \label{def:htpypushout}%
    A commutative square
    \[ \begin{tikzcd}
        W \ar[r, "f"'] \ar[d, "g"'] & U \ar[d, "p"] \\ V \ar[r, "q"] & X
        \end{tikzcd}
    \]
    in~$\Top$ is a \emph{homotopy pushout (square)} if the composition
    \[ \bigl( U \djunion (W\cross [0,1]) \djunion V \bigr) / {\sim}
            \surround\quad{
            \xrightarrow[(w,t) \mapsto f(w) \mkern1mu\text{''}]%
                        {\text{\llap{``}} u\mkern0.5mu\mapsto u,\; v\mapsto v}
            }
        \pushoutobj{U f W g V}
            \surround\quad{
            \xunderrightarrow{\pushoutmor{p f g q}}
            }
        X
    \]
    is a weak equivalence; here
    \begin{itemize}
        \item
            the first space is the \emph{double mapping cylinder} of $f$ and~$g$,
            where ${\sim}$ is the equivalence relation generated by $f(w) \sim
            (w,0)$ and $(w,1) \sim g(w)$ for $w\in W$.
            
        \item
            $\pushoutobj{U f W g V}$ is the pushout $(U\djunion V)/{\approx}$ where
            ${\approx}$ is generated by $f(w)\approx g(w)$ for $w\in W$.
            
        \item
            $\pushoutmor{p f g q}$ is the map induced by $p$ and~$q$ by the
            universal property of the pushout.
    \envEndsHere
    \end{itemize}
\end{thDef}
Similarly, homotopy pullback squares can be defined in~$\Ch$. For the
following, however, we will not need this explicitly.
\begin{thExample}
    There are homotopy pushouts
    \[
        \tikzcdset{diagrams={column sep=small, baseline=-4pt}}
        \begin{tikzcd}
                \emptyset \ar[r] \ar[d] & U \ar[d] \\ V \ar[r] & U\djunion V
        \end{tikzcd}
        \qqtextqq{and}
        \begin{tikzcd}
            X \ar[r] \ar[d] & X\cross[0,1] \ar[d] \\ X \ar[r] & X
        \end{tikzcd}
        \qqtextqq{and}
        \begin{tikzcd}
            X \ar[r] \ar[d] & \Cone(X) \ar[d] \\ \Cone(X) \ar[r] & \Sigma X
        \end{tikzcd}
    \]
    for all topological spaces $U,V,X$ (where $\Cone(X)$ is the \emph{cone
    on~$X$} and $\Sigma X$ is the \emph{suspension of~$X$}).
\end{thExample}

\begin{thRemark}[philosophy behind homotopy (co)limits] \label{rem:philosophyhtpylim}%
    Generally speaking, homotopy (co)limits make up for the fact that ordinary
    (co)limits do \emph{not} preserve the notion of weak equivalence.
    As a concrete example in~$\Top$, we consider the commutative diagram
    \[ \begin{tikzcd}[column sep=1.5cm, row sep=0.6cm]
        \point \ar[d] &  \ar[l] S^0 \ar[r, "\text{inclusion}"'] \ar[d, "\id"'] & D^1 \ar[d] \\
        \point        &  \ar[l] S^0 \ar[r]                                  & \point
        \end{tikzcd}
    \]
    and note, that even though all vertical maps are (weak) homotopy
    equivalences, the map $D^1/S^0\to\point$ that they induce
    on the ordinary pushouts of the rows is \emph{not} a (weak) homotopy
    equivalence.

    Moreover, homotopy (co)limits in a model category are strongly connected
    to \oo\-categorical (co)limits in the associated \oo\-category, i.e. the
    localization \pcref{rem:localization} at weak equivalences. For example:
    the latter compute the former\icite[Rem.~7.9.10]{b:catLR}, (certain)
    homotopy pushout squares are sent to pushout squares in the \oo\-category by
    the localization functor\icite[dual of Cor.~7.5.20]{b:catLR}, and in the
    setting of simplicially enriched categories (see \cref{rem:hcnerve}) a
    theorem by Lurie compares both notions\icite[Thm.~4.2.4.1]{b:HTT}.
\end{thRemark}

\begin{thRemark}[existence and uniqueness of homotopy pushouts]
    \label{rem:thehtpypushout}%
    Given the solid part of the diagram
    \[ \begin{tikzcd}
        W \ar[r, "f"'] \ar[d, "g"'] & U \ar[d, dashed] \\ V \ar[r, dashed] & ?
        \rlap{,}
        \end{tikzcd}
    \]
    one can always extend it to an ordinary pushout square, e.g. with
    $\mathord{?} = \pushoutobj{U f W g V}$ (in the notation of
    \cref{def:htpypushout}). However, it may happen, that there is \emph{no}
    extension to a commutative homotopy pushout square! For example, assume
    that there is a homotopy pushout
    \[ \begin{tikzcd}
        S^0 \ar[r] \ar[d] & \point \ar[d] \\ \point \ar[r] & X
        \end{tikzcd}
    \]
    and denote the corresponding double mapping cylinder by~$Z$.
    Then the weak equivalence $Z\to X$ factors through~$\point$ by definition,
    but this is impossible as $Z$~is homeomorphic to~$S^1$.
    
    On the other hand, if there are two homotopy pushouts
    \[
        \tikzcdset{diagrams={baseline=-2.5pt}}
        \begin{tikzcd}
            W \ar[r, "f"'] \ar[d, "g"'] & U \ar[d, "p"] \\ V \ar[r, "q"] & X
        \end{tikzcd}
        \qqtextqq{and}
        \begin{tikzcd}
            W \ar[r, "f"'] \ar[d, "g"'] & U \ar[d, "p'"] \\ V \ar[r, "q'"] & X'
        \end{tikzcd}
    \]
    with the same $W,U,V,f,g$, the spaces $X$ and~$X'$ must be in the same
    equivalence class that is generated by weak equivalences, because the double
    mapping cylinder of $f$ and~$g$ is weakly equivalent to both of them.
    One may thus call (the weak equivalence class of) the double mapping cylinder
    the \emph{homotopy pushout of the diagram}
    \[ \begin{tikzcd}
        V & \ar[l, "g"'] W \ar[r, "f"] & U
        \rlap{\kern1pt,}
        \end{tikzcd}
    \]
    and this is well\-defined, even if there is no extension to a homotopy
    pushout square.%
    \envEnds
\end{thRemark}

\begin{thExample} \label{ex:nonexcisiveproof}%
    In \cref{ex:nonexcisivefunctor} we claimed that there is an \oo\-categorical pushout
    \[ \tag{$\square$}\label{ex:nonexcisiveproof:square}
        \begin{tikzcd}[baseline=-2pt]
                S^0 \ar[r, "i"]\ar[d] & D^1 \ar[d, "q"'] \\
                \point \ar[r]       & D^1/S^0
        \end{tikzcd}
    \]
    in~$\Top$, where $i$~is the inclusion and $q$~the quotient map, that is
    \emph{not} a pullback. Using the language of homotopy pushouts/pullbacks
    and the references given in \cref{rem:philosophyhtpylim}, we can now give
    a proof of this:
    
    The given square~\eqref{ex:nonexcisiveproof:square} is a homotopy pushout
    and thus sent to an \oo\-pushout by the localization functor.
    However, it cannot be an \oo\-pullback: Dually to \cref{rem:thehtpypushout},
    there is a well\-defined weak equivalence class of spaces that one can call
    \emph{homotopy pullback of the diagram}
    \[ \begin{tikzcd}
        \point \ar[r] &  D^1/S^0 & \ar[l, "q"'] D^1
        \rlap{\kern1pt.}
        \end{tikzcd}
    \]
    Now if we assume, for a contradiction, that
    \eqref{ex:nonexcisiveproof:square}~is an \oo\-pullback,
    the assertion that \oo\-pullbacks compute homotopy pullbacks just means,
    that $S^0$~is in this weak equivalence class. We will now show that also
    the loop space~$\Omega S^1$, at an arbitrary base point $e\in S^1$, is
    in this equivalence class, which cannot both be true, as we would get
    \[ 2 = \card{\pi_0(S^0)} = \card{\pi_0(\Omega S^1)} = \card{\pi_1(S^1)}
         = \card{\Z}
    \]
    as cardinal numbers. To that end, consider the following diagram:
    \[ \begin{tikzcd}[column sep=1.5cm, row sep=0.6cm]
            \point \ar[r]          &  D^1/S^0        & \ar[l, "q"] D^1           \\
            \point \ar[r, "e"] \ar[u] &  S^1     \ar[u] & \ar[l, "e"'] \point \ar[u] %
        \rlap{\kern2pt.}
        \end{tikzcd}
    \]
    We can certainly find vertical maps such that they are all weak
    equivalences and such that both squares commute. Now because the homotopy
    pullback of diagrams \emph{does} respect such a \enquote{weak equivalence
    of diagrams}\icite[Prop.~13.3.4]{b:hirschhorn09} (compare
    \cref{rem:philosophyhtpylim} for ordinary limits), the homotopy pullbacks of
    the rows yield the same weak equivalence class. Using an explicit definition
    of homotopy pullback\icite[Def.~3.2.4]{b:munsonvolic15} it is then easy to
    see, that the homotopy pullback of the diagram in the bottom row is
    given by~$\Omega S^1$\icite[Ex.~3.2.10]{b:munsonvolic15}.%
    \envEnds
\end{thExample}

In the rest of this section, we will give a proof of~\cref{prop:Cexcisive1cat}
and then show how this implies \cref{thm:Cexcisive}.

\subsection{Proof of the model categorical statement}
First of all, we may reduce our attention to homotopy pushouts of a specific
type by virtue of the following three facts:
\begin{thProposition}[\titlecite{b:munsonvolic15}{Prop.~3.7.4 and Ex.~3.7.5}] \label{prop:triad}%
    Let $S'$ be a homotopy pushout square in~$\Top$. Then there exists a homotopy
    pushout square
    \[ S =
        \begin{tikzcd}[pmatrix, cramped, sep=1.65em]
            W \ar[r] \ar[d] & U \ar[d] \\ V \ar[r] & X
        \end{tikzcd}
    \]
    and a map of squares $\eta\colon S\to S'$, such that
    \begin{itemize}
        \item $U$ and~$V$ are open subsets of~$X$ and $W = U\intersect V$,
        \item all maps in~$S$ are inclusions, and
        \item $\eta$ is component-wise a weak equivalence.
    \end{itemize}
    Furthermore, any square~$S$ as above is a homotopy pushout square.%
    \envEnds
\end{thProposition}
\begin{thProposition}[\titlecite{b:hatcher02}{Prop.~4.21}] \label{prop:Chomotopical}%
    The singular chain complex functor is \emph{homotopical}, i.e. it sends weak
    equivalences in~$\Top$ to weak equivalences in~$\Ch$.
\end{thProposition}
\begin{thProposition} \label{prop:pullbackinvariant}%
    Being a homotopy pullback square in~$\Ch$ is invariant under weak
    equivalences of squares in the following sense:
    Given two squares $S$ and~$S'$
    in~$\Ch$ and a map of diagrams $S\to S'$ that is component-wise a
    weak equivalence, then $S$ is a homotopy pullback if and only if $S'$ is.
\end{thProposition}
\begin{proof}
    This is Proposition~13.3.13 in Hirschhorn's book\icite{b:hirschhorn09}.
    Its prerequisites are fulfilled because every object in~$\Ch$ is
    fibrant\icite[Thm.~2.3.11 and subsequent paragraph]{b:hovey07} and thus
    $\Ch$ is right proper\icite[Cor.~13.1.3]{b:hirschhorn09}.%
    \qedplacement{\qedNOThere}{}
\end{proof}

With this, we are left to show:
\begin{thProposition}
    Let $S$ be as in \cref{prop:triad}. Then the induced square
    \[ \begin{tikzcd}
        C(W) \ar[r] \ar[d] & C(U) \ar[d] \\ C(V) \ar[r] & C(X)
        \end{tikzcd}
    \]
    is a homotopy pullback in~$\Ch$.
\end{thProposition}
\begin{proof}
    Let $C(S)$ denote the square from the claim and let
    $\iota\colon C^{U,V}(X) \incl C(X)$ be the inclusion of the subcomplex
    generated by the singular simplices that are contained in $U$ or~$V$.
    Then it is well\-known\icite[Prop.~2.21]{b:hatcher02}
    that $\iota$ is a weak equivalence (this is the core of
    the proof of classical excision statements). Now we consider the square
    \newcommand{\ixx}[2]{
        \expandafter\newcommand\csname i#1#2\endcsname{i_{#1\!#2}}
    }%
    \ixx WU\ixx WV\ixx UX\ixx VX%
    \[  \begin{tikzcd}[pmatrix, row sep=1.9em, column sep=1.2cm]
            C(W) \ar[r, "\iWU"'] \ar[d, "\iWV"] & C(U) \ar[d, "\iUX"'] \\
            C(V) \ar[r, "\iVX"] & C^{U,V}(X)
        \end{tikzcd}
        \eqdef S^{U,V}
    \]
    where all morphisms denote the canonical injections
    and the component-wise inclusion $S^{U,V}\to C(S)$, which is
    component-wise a weak equivalence. Hence, \cref{prop:pullbackinvariant}
    applies and we can reduce to the problem of showing that~$S^{U,V}$ is a
    homotopy pullback.
    
    It is known, that $\Ch$~is a
    \emph{stable} model category\icite[Sec.~7.1]{b:hovey07}, so we may
    equivalently\icite[Rem.~7.1.12]{b:hovey07} show that $S^{U,V}$ is a
    homotopy \emph{pushout}.
    But in the model structure on~$\Ch$, all objects in~$S^{U,V}$ are
    cofibrant\icite[Lem.~2.3.6]{b:hovey07} and the morphism $\iWU$ is a
    cofibration by~\cref{lem:Cinjection} (see below), hence it
    suffices\icite[Ex.~8.8]{p:riehlverity14} to show that $S^{U,V}$ is a
    pushout. It is easily seen that this is the case if (and only if)
    \[ C(W) \xrightarrow{(\iWU,\iWV)} C(U)\dsum C(V) \xrightarrow{\iUX-\iVX}
        C^{U,V}(X) \longto 0
    \]
    is an exact sequence in~$\Ch$. But this is true, basically by construction
    of~$C^{U,V}(X)$, so we are done.%
    \qedplacement{\qedNOThere}{}
\end{proof}
This finishes the proof of~\cref{prop:Cexcisive1cat}. In the last section we
promote this model categorical result to the \oo\-categorical setting.

In the proof of the previous proposition, we made use of the following lemma,
for which we include a proof for the sake of completeness. The reader who does
not want to dive into the inner workings of the model category~$\Ch$ may safely
skip the proof and accept it as a technical fact.
\begin{thLemma}[\texorpdfstring{$C$}{C}~of injection] \label{lem:Cinjection}%
    Let $i\colon A\to X$ be an injective map in~$\Top$. Then
    $C(i)\colon C(A)\to C(X)$ is a cofibration in~$\Ch$.
\end{thLemma}
\begin{proof}
    By definition of the singular chain complex functor, $C(A)$ and~$C(X)$ are
    composed of free modules, the chain map~$C(i)$ is dimensionwise injective and maps a
    basis of~$C(A)$ to a subset of a basis of~$C(X)$. Thus, $C(i)$~is
    dimensionwise a split injection with free cokernel, hence a cofibration
    in~$\Ch$\icite[Prop.~2.3.9 and Lem.~2.3.6]{b:hovey07}.%
    \qedplacement{\qedNOThere}{}
\end{proof}

\subsection{Proof of the \pdfInfty-categorical statement}
\begin{proof}[Proof that~\cref{prop:Cexcisive1cat} implies \cref{thm:Cexcisive}]
    \newcommand{\TopOO}{\Top_\infty}%
    \newcommand{\ChOO}{\Ch_\infty}%
    \newcommand{\COO}{C_\infty}%
    \newcommand{\CxO}{\bar{C}}%
    To make this proof more transparent, let us fix the following notation:
    $\Top$~denotes the 1-category, $\TopOO$~the corresponding \oo\-category, i.e.
    the localization \pcref{rem:localization} at weak homotopy equivalences,
    similarly, $\Ch$~denotes the 1-category, $\ChOO$~the \oo\-category, i.e.
    the localization at quasi\-isomorphisms, and $C\colon\Top\to\Ch$ is the
    1-categorical functor, while $\COO\colon\TopOO\to\ChOO$ is the induced
    functor on \oo\-categories (which is well\-defined since $C$~is homotopical,
    \cref{prop:Chomotopical}).
    
    Both, $\Ch$ and~$\ChOO$ are \emph{stable},
    $\Ch$ as a model category\icite[Sec.~7.1]{b:hovey07} and
    $\ChOO$ as \oo\-category\icite[Prop.~1.3.5.9 (with Prop.~1.3.5.15)]{b:HA}.
    In particular, a square in~$\Ch$ is a homotopy pullback if and only if it is
    a homotopy pushout\icite[Rem.~7.1.12]{b:hovey07}, and a square in~$\ChOO$ is a
    pullback if and only if it is a pushout\icite[Prop.~1.1.3.4]{b:HA}.
    
    With this, we see that $\COO$~being excisive is equivalent to
    $\COO$~commuting with pushouts. Hence, it suffices to show that
    $\COO$~commutes with \emph{all} finite colimits. This follows from the dual
    of Proposition~7.5.28 in Cisinski's book\icite{b:catLR}, applied to the
    composition $\CxO\colon\nerve(\Top)\to\ChOO$ of $\nerve(C)$~and the localization functor
    $\nerve(\Ch)\to\ChOO$; for this, $\ChOO$~is made an \oo\-category with weak equivalences
    and cofibrations in a trivial way\icite[Ex.~7.5.4]{b:catLR}. To apply the
    proposition, we just need to show that $\CxO$~is a
    \emph{right exact}\icite[Def.~7.5.2]{b:catLR} functor:
    \begin{enumerate}[(i)]% Note: explicit choice of label here to match Cisinski's
        \item
            As $C(\emptyset) \cong 0$, initial objects are preserved by~$\CxO$.
            
        \item
            All morphisms in~$\ChOO$ are cofibrations and $C$~is homotopical
            \pcref{prop:Chomotopical}. Thus, $\CxO$~preserves cofibrations and
            trivial cofibrations.
            
        \item
            Lastly, we need to check that pushout squares (in~$\Top$) of the form
            \[ \begin{tikzcd}
                A \ar[r, "i"'] \ar[d] & X \ar[d] \\ A' \ar[r] & X'
                \end{tikzcd}
            \]
            with $i$~a cofibration and $A$ and~$A'$ (and thus also~$X$)
            cofibrant are sent to pushout squares by~$\CxO$. But these conditions
            ensure\icite[Ex.~8.8]{p:riehlverity14} that such a square is also a
            homotopy pushout square, which is preserved by~$C$ by combining
            \cref{prop:Cexcisive1cat} and the comment about stability of~$\Ch$
            above.
            Furthermore, by the first item and \cref{lem:Cinjection}, $C$~also
            preserves the additional properties of $i$, $A$ and $A'$,
            so the square induced by~$C$ is sent to a pushout square in~$\ChOO$%
            \icite[dual of Cor.~7.5.20]{b:catLR}.%
            \qedplacement{\\*\qedhere}{\qedhere}%
    \end{enumerate}
\end{proof}

\section{Relation between abstract and classical excision} \label{sec:classical}%
In this section, we relate abstract excision \pcref{def:excisive} to classical
excision in the form of Mayer-Vietoris sequences. To this end, we will explain
why and to some extent how an excisive functor always admits Mayer-Vietoris
sequences, coinciding with the usual one in the case of singular homology, for
example.

If our excisive functor is defined on~$\Top$,
the first thing to note is that by the last part of \cref{prop:triad} we may indeed apply the
abstract excision property in the setting $X = U\union V$ with $U$ and~$V$ open
in~$X$.

This leaves us with the task of explaining how a pullback square in an
\oo\-category leads to a long exact sequence. To this end, the first step is the
following:

\subsection{Obtaining a fiber sequence} \label{sec:fiberseq}%
Let~$C$ be an \oo\-category with finite limits and a zero object, i.e.
an object~$\ast$ of~$C$ that is both initial and terminal.
Suppose we start with the following pullback in~$C$:
\[ \begin{tikzcd}[cramped, row sep={3.48em,between origins}, column sep={3.8em,between origins}]
    W \ar[r] \ar[d] & U \ar[d] \\ V \ar[r] & X
    \rlap{.}
    \end{tikzcd}
 \]
Then from this, one can derive a fiber sequence like this:
\begin{equation} \label{eq:pbfiberseq}
    \Omega X  \to  W  \to  U \cross V
, \end{equation}
where $\Omega$ denotes the loop functor. Here, being a fiber sequence means (by
definition) being a pullback square
\[ \begin{tikzcd}
    \Omega X \ar[r] \ar[d] & W \ar[d] \\ \ast \ar[r] & U\cross V
    \end{tikzcd}
\]
and the loop functor $\Omega\colon C\to C$ assigns to an object~$X$ the pullback
of $\ast \to X \leftarrow \ast$, or in other words
it completes $\ast\to X$ to a fiber sequence $\Omega X \to \ast \to X$.
Lurie describes in detail how to construct loop and suspension
functors\icite[paragraph preceeding Rem.~1.1.2.6]{b:HA}.

For a precise treatment of the passage from pullback
to fiber sequence, we would need more technicalities about limits in
\oo\-categories. Informally, though, the basic ideas are as follows:
We have two pullback squares and a morphism~$\alpha\colon W\to U\cross V$
induced by the universal property of the product as in the right part of this
diagram:
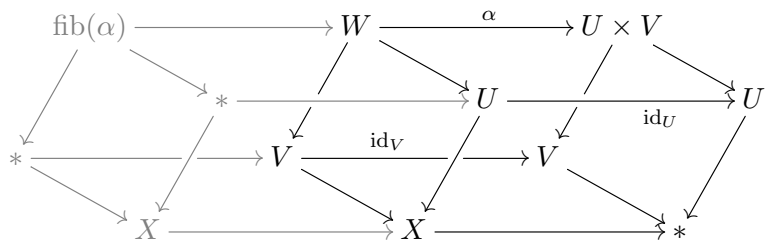
\[
\newcommand{\fib}{\operatorname{fib}}
\begin{tikzcd}[
    , column sep={1cm,between origins}
    ,    row sep={1cm,between origins}
    , matrix={name=M}
    , execute before arrows={
        \path
            node also [alias=fiAlpha] (M-1-2)
            node also [alias=fiU]     (M-2-4)
            node also [alias=fiV]     (M-3-1)
            node also [alias=fiX]     (M-4-3)
            node also [alias=pbW]     (M-1-6)
            node also [alias=pbU]     (M-2-8)
            node also [alias=pbV]     (M-3-5)
            node also [alias=pbX]     (M-4-7)
            node also [alias=prUxV]   (M-1-10)
            node also [alias=prU]     (M-2-12)
            node also [alias=prV]     (M-3-9)
            node also [alias=prPt]    (M-4-11)
        ;
    }
    , gray column/.style={/tikz/column #1/.append style={color=gray}}
    , gray column/.list={1,2,3,4}
    ]
    & \fib(\alpha) &&&& W &&&& U\cross V && \\
    &&& \ast &&&& U &&&& U \\
    \ast &&&& V &&&& V &&& \\
    && X &&&& X &&&& \ast
    \rlap{\kern1pt.}
    &
    %
    % background |_ part of the arrows:
    \arrow[from=fiAlpha, to=fiV , gray                    ]
    \arrow[from=fiV    , to=fiX , gray                    ]
    \arrow[from=pbW    , to=pbV                           ]
    \arrow[from=pbV    , to=pbX                           ]
    \arrow[from=prUxV  , to=prV                           ]
    \arrow[from=prV    , to=prPt                          ]
    \arrow[from=fiV    , to=pbV , gray                    ]
    \arrow[from=pbV    , to=prV ,       "\id_V" {pos=0.33}]
    % foreground ‾| part of the arrows:
    \arrow[crossing over, from=fiAlpha, to=fiU  , gray                     ]
    \arrow[crossing over, from=fiU    , to=fiX  , gray                     ]
    \arrow[crossing over, from=pbW    , to=pbU                             ]
    \arrow[crossing over, from=pbU    , to=pbX                             ]
    \arrow[crossing over, from=prUxV  , to=prU                             ]
    \arrow[crossing over, from=prU    , to=prPt                            ]
    \arrow[crossing over, from=fiAlpha, to=pbW  , gray                     ]
    \arrow[crossing over, from=fiU    , to=pbU  , gray                     ]
    \arrow[crossing over, from=fiX    , to=pbX  , gray                     ]
    \arrow[crossing over, from=pbW    , to=prUxV,       "\alpha"           ]
    \arrow[crossing over, from=pbU    , to=prU  ,       "\id_U"' {pos=0.68}]
    \arrow[crossing over, from=pbX    , to=prPt                            ]
\end{tikzcd}
\]
Taking fibers (also a limit!) in the horizontal direction and using the
fact that limits commute with limits\icite[dual of Lem.~5.5.2.3]{b:HTT}, as is the case
in 1-categories, we obtain that the left square is also a
pullback. Hence, by definition of the loop functor,
$\operatorname{fib}(\alpha)$ is (equivalent to) $\Omega X$.

\subsection{Applying the mapping space functor} \label{sec:Map}%
Let $C$ be a 1-category. A quite important notion in 1-category theory is the
functor
\[ \Mor_C\colon C\op \cross C \to \Set \]
that takes a pair $(x,y)$ of objects of~$C$ to the set of morphisms
$\Mor_C(x,y)$. In the literature, this is often called \enquote{Hom-functor}.

Now let $C$ be an \oo\-category.
One slogan of \oo\-category theory is: the role of sets in
1-category theory is taken by topological spaces in \oo\-category theory.
This motivates the desire to find, for a pair of objects $(x,y)$
of~$C$, not only a set of morphisms $x\to y$, but instead a \emph{space}
of such morphisms -- and indeed, this is always possible. The
following remark indicates how, but its point is rather to illustrate the
analogy between the 1- and \oo\-categorical situation than to provide the reader
with a better understanding of the space itself. Thus, it can also safely be
skipped.

\begin{thRemark}[construction of mapping spaces] \label{rem:mappingspace}%
    We consider the 1-categorical situation first. Here, the set $\Mor_C(x,y)$
    can be found as a pullback of sets as in the following pullback
    square:
    \[ \begin{tikzcd}[row sep=1.2cm]
            \color{gray}\Mor_C(x,y) \ar[r, gray] \ar[d, gray]
            & \Mor(C)
                \ar[d, "{
                    \def\ss{\scriptstyle}
                    \begin{gathered}
                        \ss (f\colon\! a\to b)\\[1pt]
                        \MTFlushSpaceAbove
                        \rotatebox[origin=c]{-90}{$\ss\mapsto$}
                        \MTFlushSpaceBelow
                        \ss (a,b)
                    \end{gathered}
                }"]
            \\
            \{0\} \ar[r, "{0 \mapsto (x,y)}"]
            & \Ob(C)\cross\Ob(C)
        \end{tikzcd}
    \]
    where $\Mor(C)$ and $\Ob(C)$ denote the class of morphisms and objects
    of~$C$, respectively.
    
    Now if $C$ is instead an \oo\-category, we can form the analogous diagram
    \[ \begin{tikzcd}[row sep=0.9cm]
            % nothing
            & C^{\mathrlap{\stdsimplex 1}}
                \ar[d, "{\left((\coface01)^*,\,(\coface00)^*\right)}"]
            \\
            \stdsimplex 0 \ar[r, "{\ast_{(x,y)}}"]
            & C^{\stdsimplex 0}\cross C^{\stdsimplex 0}
        \end{tikzcd}
    \]
    in~$\sSet$ and \emph{define} the space of morphisms from $x$ to~$y$ as
    (the geometric realization of) its pullback; here we use the following
    notation:
    \begin{itemize}
        \item
            $C^{\stdsimplex n}$ is the \emph{internal Hom} simplicial set as
            in \cref{rem:internalhom},
        \item
            $(\coface0i)^*\colon C^{\stdsimplex 1}\to C^{\stdsimplex 0}$ is the morphism
            induced by the $i$-th coface map~$\coface{0}{i}$,
            and
        \item
            $\ast_{(x,y)}$ denotes the unique morphism with
            \[ (\stdsimplex 0)_0 \ni \Id_{[0]} \mapsto (x,y) \in C_0\cross C_0
            \]
            under the identification
            $C_0\cross C_0
            \isom \Mor_{\sSet}(\stdsimplex 0, C) \cross \Mor_{\sSet}(\stdsimplex 0, C)
            = \bigl( C^{\stdsimplex 0}\cross C^{\stdsimplex 0} \bigr)_0$
            from \cref{ex:yoneda}.%
    \end{itemize}
    Dugger and Spivak\icite{p:duggerspivak11b} discuss multiple ways to
    construct mapping spaces and in particular they consider the
    construction above\icite[Prop.~1.2]{p:duggerspivak11b}.%
    \envEnds
\end{thRemark}

If properly organized\icite[Sec.~5.8]{b:cisinski19}, one can even define a
functor of \oo\-categories
\[ \Map_C\colon C\op\cross C \to \Top  ,\]
which we call $\Map_C$ to distinguish it from the 1-categorical concept. (Note,
however, that $\Map_C(x,y)$ will in general only be weakly equivalent to the
space that we constructed in \cref{rem:mappingspace}. Also, in Cisinski's book
the target of the mapping space functor is~\enquote{$\mathscr{S}$}, so to get
one to~$\Top$, we have to compose with an equivalence between those two
\oo\-categories\icite[Rem.~7.8.10 and Rem.~7.8.11]{b:cisinski19}.)
For this functor one can show:
\begin{thProposition}[\titlecite{b:cisinski19}{Cor.~6.3.5}] \label{prop:mapcontinuous}%
    Let $C$~be a locally small \oo\-category and let $x$ be an object of~$C$.
    Then the functor $\Map_C(x,\blank)\colon C\to\Top$ preserves limits.
\end{thProposition}
In particular, since these are defined via pullbacks, $\Map_C(x,\blank)$
maps fiber sequences to fiber sequences. This gives us a method to convert the
fiber sequence~\eqref{eq:pbfiberseq} from \cref{sec:fiberseq} to one in~$\Top$: for any
test object~$S$ of~$C$ we obtain
\begin{equation} \label{eq:mapfiberseq}
    \Map_C(S,\Omega X)  \to  \Map_C(S,W)  \to  \Map_C(S,U \cross V)
, \end{equation}
a fiber sequence of spaces, and canonically also of
pointed spaces by virtue of zero maps.

\subsection{The Mayer-Vietoris sequence} \label{sec:mayervietorishomotopy}%
As is well\-known, a fiber sequence $F\to X\to Y$ of pointed spaces
always yields a long exact sequence of homotopy groups:
\[ \cdots \to
    \pi_n(F) \to \pi_n(X) \to \pi_n(Y) \to \pi_{n-1}(F) \to \cdots
.\]
Using this on the fiber sequence~\eqref{eq:mapfiberseq} from the last section,
again the fact that $\Map_C(S,\mspace{0mu plus 0mu minus 0.6mu}\blank\mspace{0mu plus 0mu minus 0.8mu})$ % l-TODO: tempering necessary?
preserves limits \pcref{prop:mapcontinuous},
and well\-known facts about the homotopy groups,
we get the following exact sequence:
\begin{multline} \label{eq:htpygroupsofmaples}
    \cdots \to
    \pi_{n+1}\Map_C(S,X) \to \pi_n\Map_C(S,W) \\ \to
    \pi_n\Map_C(S,U) \cross \pi_n\Map_C(S,V)  \to
    \pi_n\Map_C(S,X) \to \cdots
. \end{multline}
In this sense, one can always extract Mayer-Vietoris type sequences
from a pullback square.

\subsection{Example: singular homology}
We shall now explain, how the long exact sequence~\eqref{eq:htpygroupsofmaples}
from \cref{sec:mayervietorishomotopy}
about homotopy groups also yields the familiar Mayer-Vietoris
sequence of \emph{homology} groups. For this we have to use the
\enquote{correct} object~$S$ plus some additional arguments.
In the following, we reuse the notation from the beginning of
\cref{sec:singular}.

Let us first recapitulate the situation: We have a space $X = U\union V$ with $U$
and~$V$ open in~$X$, $W = U\intersect V$, and $C\colon\Top\to\Ch$~the singular chain
complex functor with arbitrary constant coefficients over some ring~$k$.
From the pushout formed by $U\intersect V$, $U$, $V$ and~$X$, we get the according
pullback of chain complexes. We now apply the previous material, where for the test
complex~$S$ we choose the complex $k[0]$ that has $k$ in degree~$0$ and is
trivial elsewhere. Note, that $k[0]$ is just one usual notation for this and no
connection to the preorder category~$[0]$ from \cref{def:simplicialset} is
intended.
Looking at the sequence~\eqref{eq:htpygroupsofmaples}, we then get a lot of
terms of the form
\[ \pi_n\Map_{\Ch}\bigl(k[0],C(Z)\bigr) \]
for $n\in\N$ and $Z$~one of $W$, $U$, $V$, $X$, which can be computed via the
following:
\begin{thProposition} \label{prop:htpygroupofMap}%
    Let $c$ be a chain complex and let~$n\in\N$. Then:
    \[ \pi_n\Map_{\Ch}(k[0],c) \isom H_n(c) \]
    where the right hand side is the $n$-th homology of~$c$.
\end{thProposition}
\begin{proof}
    We split the proof as follows:
    \begin{enumerate}[(1)]
        \item \label{prop:htpygroupofMap:proof:dgCh}%
            In the first part, we show the claim with $\Ch$~replaced
            by a weaker localization, namely by the localization at
            chain homotopy equivalences (instead of all quasi\-isomorphisms).
            
        \item \label{prop:htpygroupofMap:proof:bousfield}%
            In the second part, we pass from this intermediate \oo\-category of
            chain complexes to~$\Ch$. In contrast to the first part, which
            provides the link between mapping spaces and homology, this part
            can be regarded as a technicality which shows that mapping out
            of~$k[0]$ provides the same homotopical view on other objects
            in both localizations.
    \end{enumerate}
    
    As both steps use technology that was not previously introduced in this
    article, we do \emph{not} expect the reader to understand all the details.
    However, we try hard to conceptually explain what happens and provide
    necessary literature pointers.
    
    \emph{Ad~\ref{prop:htpygroupofMap:proof:dgCh}.}
    \newcommand{\geomrealization}[1]{\abs*{#1}}
    \newcommand{\DK}{\operatorname{K}}
    \newcommand{\truncation}{\tau_{\geq0}}%
    We can\icitekerodon{00PE} view the 1-category~$\Ch$ as a
    \emph{DG-category}, which we denote by~$\dgCh$; this
    means\icitekerodon{00P9} that $\dgCh$ is enriched over chain complexes of
    Abelian groups, i.e.
    $\Mor_{\dgCh}(c',c)$ is an object of~$\Ch_{\Z}$ for all chain
    complexes~$c',c$.
    Then the so-called \emph{DG-nerve}\icitekerodon{00PK}, a gadget similar in
    spirit to the homotopy coherent nerve from \cref{rem:hcnerve}, provides us
    with an \oo\-category $\nervedg(\dgCh)$, which can be seen as the
    localization \pcref{rem:localization} of~$\Ch$ at chain homotopy
    equivalences\icite[Prop.~1.3.4.5]{b:HA}. Mapping spaces in \oo\-categories
    obtained via~$\nervedg$ are particularly tractable%
    \icite[Rem.~1.3.1.12]{b:HA}, hence
    \[ \Map_{\nervedg(\dgCh)}(k[0],c)
        \equiv \geomrealization{ \DK\bigl(\truncation \Mor_{\dgCh}(k[0],c)\bigr) }
    , \]
    where
    \begin{itemize}[itemsep={3pt plus 1pt minus .5pt}]
        \item $\geomrealization{\blank}$ denotes the geometric realization of a
            simplicial set (in this case: after forgetting the group structure),

        \item $\DK$ sends a non-negative chain complex to its associated
            simplicial Abelian group under the \emph{Dold-Kan correspondence}%
            \icite[Sec.~8.4]{b:weibel95},
            and
            
        \item $\truncation$ is the good truncation
            functor\icite[{}1.2.7]{b:weibel95} that essentially
            forces a chain complex to be non-negative.
    \end{itemize}
    Thus, we see
    \begin{align*}
        \pi_n\Map_{\nervedg(\dgCh)}(k[0],c)
        &\isom \pi_n\DK\truncation \Mor_{\dgCh}(k[0],c)  \\
        &\isom H_n\bigl(\Mor_{\dgCh}(k[0],c)\bigr)       \\
        &\isom H_n(c)
    , \end{align*}
    because taking the homotopy group of a geometric realization is isomorphic
    to the simplicial homotopy group\icite[p.~60]{b:goerssjardine09} and
    the latter is isomorphic to the homology of the associated chain complex%
    \icite[Thm.~8.4.1]{b:weibel95}; finally, $\Mor_{\dgCh}(k[0],c) \isom c$
    is easily seen by expanding the definitions.
    This shows the claim for~$\nervedg(\dgCh)$ instead of~$\Ch$.
    
    \emph{Ad~\ref{prop:htpygroupofMap:proof:bousfield}.}
    To pass from~$\nervedg(\dgCh)$ to~$\Ch$ (i.e. from localization at chain
    homotopy equivalences to localization at all quasi-isomorphisms), we
    additionally use the so-called \emph{Hurewicz model structure} on~$\Ch$,
    which has chain homotopy equivalences as weak equivalences. It
    can be combined with the projective model structure to yield a \emph{mixed
    model structure}\icite[Sec.~17.3]{b:mayponto12}. All three model structures are
    discussed in May and Ponto's book under the names
    \enquote{\textit{h}-/\textit{q}-/\textit{m}-model
    structure}\icite[Sec. 18.3, 18.4 and~18.6]{b:mayponto12}.
    Both, projective and mixed, have the quasi-isomorphisms as weak
    equivalences, but the latter has the advantage of being a right Bousfield
    localization of the Hurewicz model structure. This can be used to see that
    the restriction of a mapping space functor in the sense of
    Cisinski\icite[(dual/cofibrant version of)~7.10.7]{b:catLR}
    for the Hurewicz model structure yields a mapping space functor for the
    mixed model structure.
    Together with the fact that $k[0]$~is cofibrant in all three
    model structures%
    \icite[Thm.~18.3.1/Prop.~18.5.2\,(iii)/Sec.~18.6]{b:mayponto12},
    we get
    \[ \Map_{\nervedg(\dgCh)}(k[0],c) \equiv \Map_{\Ch}(k[0],c)
    .
    \qedplacement{\qedafterdisplay}{\qedhere} \]
\end{proof}
\noindent%
If we apply this proposition to the long exact sequence~\eqref{eq:htpygroupsofmaples} that we get from
\cref{sec:mayervietorishomotopy} for $S = k[0]$, we indeed arrive at the
familiar long exact sequence:
\[ \cdots
    \to H_{n+1}\bigl(C(X)\bigr)
    \to H_n\bigl(C(U\intersect V)\bigr)
    \to H_n\bigl(C(U)\bigr)\cross H_n\bigl(C(V)\bigr)
    \to H_n\bigl(C(X)\bigr)
    \to \cdots
. \]

\appendix
\clearpage
\begingroup
\setlength{\hfuzz}{10pt}% l-TODO: check if necessary
\printbibliography
\endgroup
\end{document}